\documentclass[11pt,a4paper]{amsart}
\usepackage[english]{babel}
\usepackage{amsmath}
\usepackage{amsfonts}
\usepackage{amssymb}
\usepackage{graphicx}
\usepackage{lmodern}

\usepackage{geometry}
\newtheorem{assumption}{Assumption}[section]

\newtheorem{lemma}{Lemma}[section]

\newtheorem{theorem}{Theorem}[section]

\newtheorem{remark}{Remark}[section]

\newtheorem{corollary}{Corollary}[section]

\DeclareMathOperator{\erf}{erf}
\newcommand{\R}{\mathbb R}
\newcommand{\N}{\mathbb N}
\newcommand{\E}{\mathbb{E}}

\newcommand{\X}{\bar{X}}

\newcommand{\C}{\mathcal{C}}
\renewcommand{\P}{\mathcal{P}}
\newcommand{\law}{\text{law}}
\newcommand{\e}{\epsilon}
\newcommand{\ve}{\varepsilon}
\newcommand{\tol}{\texttt{tol}}
\newcommand{\supp}{\text{supp}}

\newcommand{\om}{\omega_f^\alpha}
\newcommand{\lt}{\left}
\newcommand{\rt}{\right}
\newcommand{\pa}{\partial}
\newcommand{\m}{m}

\begin{document}

\title[Consensus-based Optimization]{An analytical framework for a consensus-based\\ global optimization method}

\author[Carrillo]{Jos\'{e} A. Carrillo}
\address[Jos\'{e} A. Carrillo]{\newline Department of Mathematics, Imperial College London, 
    \newline London SW7 2AZ, United Kingdom}
\email{carrillo@imperial.ac.uk}

\author[Choi]{Young-Pil Choi}
\address[Young-Pil Choi]{\newline Department of Mathematics and Institute of Applied Mathematics, Inha University,
    \newline Incheon 402-751, Republic of Korea}
\email{ypchoi@inha.ac.kr}

\author[Totzeck]{Claudia Totzeck}
\address[Claudia Totzeck]{\newline Department of Mathematics, Technische Universit\"at Kaiserslautern, 
\newline Erwin-Schr\"odinger-Strasse, 67663 Kaiserslautern, Germany}
\email{totzeck@mathematik.uni-kl.de}

\author[Tse]{Oliver Tse}
\address[Oliver Tse]{\newline Department of Mathematics and Computer Science, Eindhoven University of Technology, 
\newline P.O. Box 513, 5600MB Eindhoven, The Netherlands}
\email{o.t.c.tse@tue.nl}

\begin{abstract}
In this paper we provide an analytical framework for investigating the efficiency of a consensus-based model for tackling global optimization problems. This work justifies the optimization algorithm in the mean-field sense showing the convergence to the global minimizer for a large class of functions. Theoretical results on consensus estimates are then illustrated by numerical simulations where variants of the method including nonlinear diffusion are introduced.
\end{abstract}



\maketitle
\section{Introduction}\label{sec:intro}

Over the last decades, individual-based models (IBMs) have been widely used in the investigation of complex systems that manifest self-organization or collective behavior. Examples of such complex systems include swarming behavior, crowd dynamics, opinion formation, synchronization, and many more, that are present in the field of mathematical biology, ecology and social dynamics, see for instance \cite{bellomo2013,MR3304346,MR2596552,MR2744704,MR2295620,MR2425606,MR3143990,MR3274797,MR2247927}, and the references therein.

In the field of global optimization, IBMs may be found in a class of {\em metaheuristics} (e.g.~evolutionary algorithms,\cite{aarts1988simulated,back1997handbook,reeves2003genetic} and swarm intelligence,\cite{dorigo2005ant,kennedy2010particle}). They play an increasing role in the design of fast algorithms to provide sufficiently good solutions in tackling hard optimization problems, which includes the traveling salesman problem that is known to be NP hard. Metaheuristics, in general, may be considered as high level concepts for exploring search spaces by using different strategies, chosen in such a way, that a dynamic balance is achieved between the exploitation of the accumulated search experience and the exploration of the search space.\cite{blum2003metaheuristics} Notable metaheuristics for global optimization include, for example, the Ant Colony Optimization, Genetic Algorithms, Particle Swarm Optimization and Simulated Annealing,\cite{holley1988simulated,holley1989asymptotics} all of which are stochastic in nature.\cite{bianchi2009survey} Despite having to stand the test of time, a majority of metaheuristical methods lack the proper justification of its efficacy in the mathematical sense---the universal intent of research in the field is to ascertain whether a given metaheuristic is capable of finding an optimal solution when provided with sufficient information. Due to the stochastic nature of metaheuristics, answers to this question are nontrivial, and they are always probabilistic. 

Recently, the use of opinion dynamics and consensus formation in global optimization has been introduced in,\cite{doi:10.1142/S0218202517400061} where the authors showed substantial numerical and partial analytical evidence of its applicability to solving multi-dimensional optimization problems of the form
\[
 \min\nolimits_{x\in\Omega}\,f(x),\qquad\text{$\Omega\subset\R^d$ a domain},
\]
for a given cost function $f\in\C(\Omega)$ that achieves its global minimum at a unique point in $\Omega$. Without loss of generality, we may assume $f$ to be positive and defined on the whole $\R^d$ by extending it outside $\Omega$ without changing its global minimum.

Throughout the manuscript, we will use the notations $\underline f = \inf f$, $\overline f = \sup f$, and 
\[
 x_* = \arg\min f,\qquad f_* = f(x_*).
\]
As we assume to have a unique global minimizer, it holds $f_* = \underline f$.

The optimization algorithm involves the use of multiple agents located within the domain $\Omega$ to dynamically establish a consensual opinion amongst themselves in finding the global minimizer to the minimization problem, while taking into consideration the opinion of all active agents. First order models for consensus have been studied within the mathematical community interested in granular materials and swarming that lead to aggregation-diffusion and kinetic equations, which have nontrivial stationary states or flock solutions (cf.~\cite{MR3035983,MR3251743,MR3158478,MR3331178}, and the references therein). They are also common tools in control engineering to establish consensus in graphs (cf.~\cite{MR2086916,MR3045811}). 

In order to achieve the goal of optimizing a given function $f(x)$, we consider an interacting stochastic system of $N\in\N$ agents with position $X_t^i\in\R^d$, described by the system of stochastic differential equations
\begin{subequations}\label{eq:particle}
\begin{align}
 d X_t^i &= -\lambda(X_t^i - \m_t)\,dt + \sigma|X_t^i - \m_t|dW_t^i,  \label{eq:particle_a}\\
 \m_t &= \displaystyle\sum_{i=1}^N X_t^i\,\left(\frac{\om(X_t^i)}{\sum_{j=1}^N \om(X_t^j)}\right), \label{eq:particle_b}
\end{align}
\end{subequations}
with $\lambda,\sigma>0$, where $\om$ is a weight, which we take as $\om(x) = \exp(-\alpha f(x))$ for some appropriately chosen $\alpha>0$. Notice that \eqref{eq:particle} resembles a geometric Brownian motion, which drifts towards $\m_t\in\R^d$. This system is a simplified version
of the algorithm introduced in,\cite{doi:10.1142/S0218202517400061} while keeping the essential ingredients and mathematical difficulties. The first term in \eqref{eq:particle_a} imposes a global relaxation towards a position determined by the behavior of the normalized moment given by $\m_t$, while the diffusion term tries to concentrate again around the behavior of $\m_t$. In fact, agents with a position differing a lot from $\m_t$ are diffused more. Hence they explore a larger portion of the landscape of the graph of $f(x)$, while the explorer agents closer to $\m_t$ diffuse much less. The normalized moment $\m_t$ is expected to dynamically approach the global minimum of the function $f$, at least when $\alpha$ is large enough, see below. This idea is also used in simulated annealing algorithms.\cite{holley1988simulated,holley1989asymptotics} The well-posedness of this system will be thoroughly investigated in Section~\ref{sec:microscopic}.

For the solution $X_t^1,\ldots,X_t^N$, $N\in\N$ of the particle system \eqref{eq:particle}, we can consider its empirical measure given by
\[
\rho_t^N = \frac{1}{N}\sum\nolimits_{i=1}^N \delta_{X_t^i},
\]
where $\delta_x$ is the Dirac measure at $x\in\R^d$. Observe that $\m_t$ may be re-written in terms of $\rho_t^N$, i.e.,
\[
 \m_t = \frac{1}{\|\om\|_{L^1(\rho_t^N)}}\int x\,\om d\rho_t^N =: \m_f[\rho_t^N].
\]
Therefore \eqref{eq:particle_a} may be formulated as
\[
 d X_t^i = -\lambda(X_t^i - \m_f[\mu_t^N])\,dt + \sigma|X_t^i - \m_f[\mu_t^N]|dW_t^i,
\]
for which we postulate the limiting ($N\to\infty$) {\em nonlinear process} $\X_t$ to satisfy
\begin{subequations}\label{eq:nonlinear}
\begin{gather}
  d\X_t = -\lambda(\X_t - \m_f[\rho_t])\,dt + \sigma|\X_t - \m_f[\rho_t]|dW_t, \label{eq:nonlinear_sde}\\
  \m_f[\rho_t] = \int x\,d\eta_t^\alpha,\qquad \eta_t^\alpha = \om \rho_t / \|\om\|_{L^1(\rho_t)},\qquad \rho_t=\law(\X_t),\label{eq:nonlinear_sde_b}
\end{gather}
\end{subequations}
subject to the initial condition $\law(\X_0)=\rho_0$. We call $\eta_t^\alpha$ the {\em $\alpha$-weighted measure}. The measure $\rho_t=\law(\X_t)\in\P(\R^d)$ is a Borel probability measure, which describes the evolution of a one-particle  distribution resulting from the mean-field limit. 

The (infinitesimal) generator corresponding to the nonlinear process \eqref{eq:nonlinear_sde} is given by 
\begin{align*}
L\varphi = \kappa \Delta \varphi - \mu\cdot\nabla\varphi, \quad \text{for $\varphi \in \mathcal{C}^2_c(\R^d)$,}
\end{align*}
with drift and diffusion coefficients
\[
\mu_t =\lambda(x-\m_f[\rho_t]),\qquad \kappa_t = (\sigma^2/2)|x-\m_f[\rho_t]|^2,
\]
respectively. Therefore, the Fokker--Planck equation associated to \eqref{eq:nonlinear} reads
\begin{align}\label{eq:fokker-planck}
\partial_t \rho_t = \Delta(\kappa_t\rho_t) + \nabla\cdot(\mu_t\rho_t),\qquad \lim\nolimits_{t\to 0}\rho_t = \rho_0,
\end{align}
where $\rho_t\in\P(\R^d)$ for $t\ge 0$ satisfies \eqref{eq:fokker-planck} in the weak sense.
Notice that the Fokker--Planck equation \eqref{eq:fokker-planck} is a nonlocal, nonlinear degenerate drift-diffusion equation, which makes its analysis a nontrivial task. Its well-posedness will be the topic of Section~\ref{sec:mesoscopic}. 

We recall from \cite{doi:10.1142/S0218202517400061} (cf.~\cite{dembo2009large}), that for any $\rho\in\P(\R^d)$, $\om\rho$ satisfies the well-known {\em Laplace principle}:
\[
 \lim_{\alpha\to\infty} \left( -\frac{1}{\alpha}\log\left(\int e^{-\alpha f} d\rho \right)\right) = \underline f|_{\text{supp}(\rho)} > 0,
\]
Therefore, if $f$ attains its minimum at a single point $x_*\in\text{supp}(\rho)$, then the $\alpha$-weighted measure $\eta_t^\alpha\in\P(\R^d)$ assigns most of its mass to a small region around $x_*$  and hence it approximates a Dirac distribution $\delta_{x_*}$ at $x_*\in\R^d$ for large $\alpha\gg1$. Consequently, the first moment of $\eta_t^\alpha$, given by $\m_f[\rho]$, provides a good estimate of $x_*=\arg\min f$. 
Using this fact, we proceed to give justifications for the applicability of the microscopic system \eqref{eq:particle} as a tool for solving global optimization problems, via its mean-field counterpart.

\begin{figure}[ht!]
\centering \includegraphics[width=0.5\textwidth]{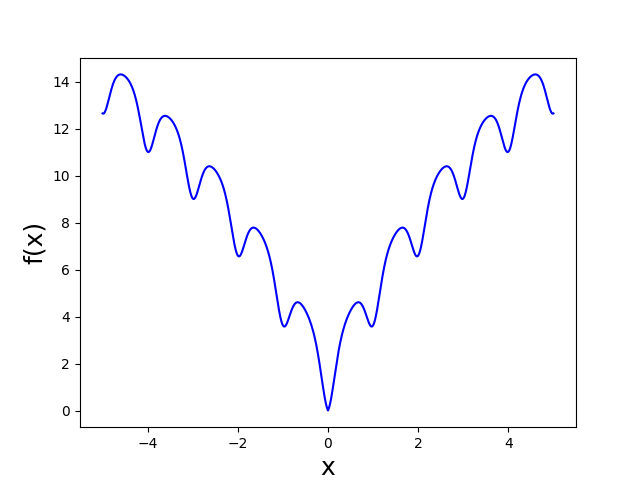}
\label{fig:Preprocessing}
\caption{The Ackley function is a well-known benchmark for global optimization problems due to its various local minima and the unique global minimum.}
\end{figure}  

Our main results show in Section~\ref{sec_lt} that mild assumptions on the regularity of the objective function $f$, $f\in W^{2,\infty}(\R^d)$, one obtains a uniform consensus as the limiting measure $(t\to\infty)$ corresponding to \eqref{eq:fokker-planck}, i.e.,
\[
 \rho_t \longrightarrow \delta_{\tilde x}\quad\text{as}\;\; t\to\infty,
\]
for some $\tilde x\in\R^d$ possibly depending on the initial density $\rho_0$. It is also shown that this convergence happens exponentially in time. Moreover, under the same assumptions on $f$, the point of consensus $\tilde x$ may be made arbitrarily close to $x_*=\arg\min f$ by choosing $\alpha\gg 1$ sufficiently large, which is the main goal for global optimization. Our regularity assumptions allow for complicated landscapes of objective functions with as many local minimizers as you want but with a well defined unique global minimum, see for instance the Ackley function--a well-known benchmark for global optimization problems\cite{askari2013large}---used in \cite{doi:10.1142/S0218202517400061} and depicted in Figure \ref{fig:Preprocessing}. {\it Up to our knowledge, this work shows for the first time the convergence of an agent-based stochastic scheme for global optimization with mild assumptions on the regularity of the cost function.}
We conclude the paper with an extension of the Fokker--Planck equation \eqref{eq:fokker-planck} to include nonlinear diffusion of porous medium type and provide numerical evidence for consensus formation in the one dimensional case. For this reason, we introduce an equivalent formulation of the mean-field equation in terms of the pseudo-inverse distribution $\chi_t(\eta) = \inf\{x\in\R\,|\, \rho_t((-\infty,x])>\eta\}$. We also compare the microscopic approximation corresponding to the porous medium type Fokker--Planck equation with the original consensus-based microscopic system \eqref{eq:particle} and the proposed algorithm in \cite{doi:10.1142/S0218202517400061}, showcasing the exponential decay rate of the error in suitable transport distances towards the global minimizer.

\section{Well-posedness of the Microscopic Model}\label{sec:microscopic}

We begin by studying the existence of a unique process $\{{\bf X}_t^{(N)}\,|\,t\ge 0\}$ with ${\bf X}^{(N)}:=(X^{(1,N)},\ldots,X^{(N,N)})^\top$, satisfying the consensus-based optimization scheme \eqref{eq:particle}, and write, for an arbitrary but fixed $N\in\N$, system \eqref{eq:particle} as
\begin{align}\label{eq:system}
 d{\bf X}_t^{(N)} = -\lambda {\bf F}_N({\bf X}_t^{(N)})\,dt + \sigma {\bf M}_N({\bf X}_t^{(N)})\,d{\bf W}_t^{(N)},
\end{align}
where ${\bf W}=(W^{(1,N)},\ldots,W^{(N,N)})^\top$ is the standard Wiener process in $\R^{Nd}$, and
\begin{gather*}
 {\bf F}_N({\bf X}) = (F_N^1({\bf X}),\ldots,F_N^N({\bf X}))^\top\in\R^{Nd},\\
 \text{where}\;\;F_N^i({\bf X}) = \frac{\sum\nolimits_{j\ne i} (X^i-X^j)\,\om(X^j)}{\sum\nolimits_j \om(X^j)}\in\R^d, \\
 {\bf M}_N({\bf X}) = \text{diag}(|F_N^1({\bf X})|\mathbb{I}_d,\ldots,|F_N^N({\bf X})|\mathbb{I}_d) \in \R^{Nd\times Nd}.
\end{gather*}
At this point, we make smoothness assumptions regarding our cost function $f$.

\begin{assumption}
 The cost function $f\colon\R^d\to\R$ is locally Lipschitz continuous.
\end{assumption}

 Under these conditions on $f$, we easily deduce that $F_N^i$, $1\le i\le N$, is locally Lipschitz continuous and has linear growth. Consequently, ${\bf F}_N$ and ${\bf M}_N$ are locally Lipschitz continuous and have linear growth. To be more precise, we obtain the following result.

\begin{lemma}\label{lem:localLip}
 Let $N\in\N$, $\alpha,k>0$ be arbitrary. Then for any ${\bf X},{\bf \hat X}\in\R^{Nd}$ with $|{\bf X}|,|{\bf \hat X}|\le k$ and all $i =1, \dots,N$ it holds
 \begin{align*}
  |F_N^i({\bf X}) -  F_N^i({\bf \hat X})| &\le |X^i - \hat X^i| + \lt(1+\frac{2c_k}{N}\sqrt{N|\hat X^i|^2 + |{\bf  \hat X}|^2}\;\rt)|{\bf X} - {\bf \hat X}|,\\
  |F^i_N({\bf X})| &\le |X^i| + |{\bf X}|,
 \end{align*}
 where $c_k= \alpha\|\nabla f\|_{L^\infty(B_k)}\exp(\alpha \|f-\underline f\|_{L^\infty(B_k)})$ and $B_k=\{x\in\R^d\;|\; |x|\le k\}$.
\end{lemma}
\begin{proof}
 Let ${\bf X},{\bf \hat X}\in\R^{Nd}$ with $|{\bf X}|,|{\bf \hat X}|\le k$ for some $k\ge 0$ and $i \in \{ 1,\dots, N\}$ be arbitrary. Then
 \[
  F_N^i({\bf X}) -  F_N^i({\bf \hat X}) = \frac{\sum\nolimits_{j\ne i} (X^i-X^j)\,\om(X^j)}{\sum\nolimits_j \om(X^j)} - \frac{\sum\nolimits_{j\ne i} (\hat X^i- \hat X^j)\,\om( \hat X^j)}{\sum\nolimits_j \om(\hat X^j)} = \sum\nolimits_{\ell=1}^3 I_\ell,
 \]
 where the terms $I_\ell$, $\ell=1,2,3$, are given by
 \begin{align*}
  I_1 &= \frac{\sum\nolimits_{j\ne i} (X^i - \hat X^i + \hat X^j-X^j)\,\om(X^j)}{\sum\nolimits_j \om(X^j)},\\
  I_2 &= \frac{\sum\nolimits_{j\ne i} ( \hat X^i-\hat X^j)\lt(\om(X^j) - \om(\hat X^j)\rt)}{\sum\nolimits_j \om(X^j)},\\
  I_3 &= \sum\nolimits_{j\ne i} (\hat X^i- \hat X^j)\,\om( \hat X^j) \frac{\sum\nolimits_{j} \lt(\om(\hat X^j) - \om(X^j)\rt)}{\sum\nolimits_j \om(X^j) \sum\nolimits_j \om(\hat X^j)},
 \end{align*}
 that may easily be estimated by
 \begin{align*}
  |I_1| &\le |X^i - \hat X^i|  + |{\bf X} - {\bf \hat X}|,\\
  |I_2| &\le \frac{\sqrt{2}c_k}{N}|{\bf X} - {\bf \hat X}|\sqrt{ N|\hat X^i|^2 + |{\bf \hat X}|^2},\\
  |I_3| &\le \frac{c_k}{N}|{\bf X} - {\bf \hat X}|\sqrt{N|\hat X^i|^2 + |{\bf \hat X}|^2}.
 \end{align*}
Putting all these terms together yields the required estimate.

As for the estimate of $|F^i_N({\bf X})|$, we easily obtain
\[
|F^i_N({\bf X})| = \left|X^i - \frac{\sum\nolimits_{j} X^j\,\om(X^j)}{\sum\nolimits_j \om(X^j)}\right|\leq |X^i| + |{\bf X}|.
\]
As $i$ was chosen arbitrary, this concludes the result.
\end{proof}

Due to Lemma~\ref{lem:localLip}, we may invoke standard existence results of strong solutions for \eqref{eq:system}.\cite{durrett1996stochastic}

\begin{theorem}
 For each $N\in\N$, the stochastic differential equation \eqref{eq:system} has a unique strong solution $\{{\bf X}_t^{(N)}\,|\,t\ge 0\}$ for any initial condition ${\bf X}_0^{(N)}$ satisfying $\E|{\bf X}_0^{(N)}|^2<\infty$.
\end{theorem}
\begin{proof}
 As mentioned above, we make use of a standard result on existence of a unique strong solution. To this end, we show the existence of a constant $b_N>0$, such that
 \begin{align}\label{eq:estimate_N}
  -2\lambda{\bf X}\cdot {\bf F}_N({\bf X}) + \sigma^2\text{trace}({\bf M}_N {\bf M}_N^\top)({\bf X}) \le b_N|{\bf X}|^2.
 \end{align}
 Indeed, since the following inequalities hold:
 \begin{align*}
   -X^i\cdot F_N^i({\bf X}) &= -X^i\cdot \frac{\sum\nolimits_{j\ne i} (X^i - X^j)\om(X^j)}{\sum\nolimits_j \om(X^j)} \le -|X^i|^2 + |X^i||{\bf X}|,\\
  |F_N^i({\bf X})|^2 &= \left|\frac{\sum\nolimits_{j\ne i} (X^i - X^j)\om(X^j)}{\sum\nolimits_j \om(X^j)}\right|^2  \le 2\lt(|X^i|^2+ |{\bf X}|^2\rt),
 \end{align*}
 we conclude that
 \begin{align*}
  -2\lambda {\bf X}\cdot {\bf F}_N({\bf X}) + \sigma^2\text{trace}({\bf M}_N {\bf M}_N^\top)({\bf X}) &= \sum\nolimits_i \left(-2\lambda  X^i \cdot F_N^i({\bf X}) + d\sigma^2|F_N^i({\bf X})|^2\right) \\
  &\hspace*{-4.5em} \le \sum\nolimits_i 2\lambda\lt(-|X^i|^2 + |X^i||{\bf X}| \rt)+ 2d\sigma^2\lt(|X^i|^2+ |{\bf X}|^2\rt) \\
  &\hspace*{-4.5em} \le 2\left(\lambda\sqrt{N} + 2d\sigma^2N\right)|{\bf X}|^2 =: b_N|{\bf X}|^2.
 \end{align*}
 Along with the local Lipschitz continuity and linear growth of ${\bf F}_N$ and ${\bf M}_N$, we obtain the assertion by applying Theorem~3.1 of \cite{durrett1996stochastic}.
\end{proof}

\begin{remark}
 In fact, the estimate \eqref{eq:estimate_N} yields a uniform bound on the second moment of ${\bf X}_t$. Indeed, by application of the It\^o formula, we obtain
 \begin{align*}
  \frac{d}{dt}\E|{\bf X}_t^{(N)}|^2 &= -2\lambda\E[ {\bf X}_t \cdot {\bf F}_N({\bf X}_t^{(N)})] + \sigma^2\E[\text{trace}({\bf M}_N {\bf M}_N^\top)({\bf X}_t^{(N)})] \\
  &\le b_N\,\E|{\bf X}_t^{(N)}|^2.
 \end{align*}
 Therefore, the Gronwall inequality yields
 \[
  \E|{\bf X}_t^{(N)}|^2 \le e^{b_Nt}\E|{\bf X}_0^{(N)}|^2\qquad\text{for all\, $t\ge 0$},
 \]
 i.e., the solution exists globally in time for each fixed $N\in \N$. 
\end{remark}
 Unfortunately, for the mean-field limit ($N\to\infty$) we lose control of the previous bound, since $b_N\to\infty$ as $N\to\infty$. Therefore, we will need a finer moment estimates on ${\bf X}^{(N)}$, which we establish at the end of the following section (cf.~Lemma~\ref{lem:moment_particle}). 

%
%
%

\section{Well-posedness of the Mean-field Equation}\label{sec:mesoscopic}

In this section, we provide the well-posedness of the nonlocal, nonlinear Fokker--Planck equation \eqref{eq:fokker-planck}. Since we will be working primarily with Borel probability measures on $\R^d$ with finite second moment, we provide its definition for the readers convenience. 

We denote the space of Borel probability measures on $\R^d$ with finite second moment by
	\[
	\mathcal{P}_2(\R^d):=\left\lbrace \mu\in\P(\R^d) \quad\text{such that}\quad \int_{\R^d}|z|^2 \mu(dz)<\infty\right\rbrace.
	\]
that we equip with the 2-Wasserstein distance $W_2$ defined by 
\[
W_2^2(\mu,\hat \mu) =\inf\, \biggl\{ \int_{\R^d\times\R^d} |z-\hat z|^2\pi(d z,d\hat z)\;,\; \pi\in\Pi(\mu,\hat \mu)\biggr\},\qquad \mu,\hat \mu\in\P_2(\R^d),
\]
where $\Pi(\mu,\hat \mu)$ denotes the collection of all Borel probability measures on $\R^d\times\R^d$ with marginals $\mu$ and $\hat \mu$ on the first and second factors respectively. The set $\Pi(\mu,\hat \mu)$ is also known as the set of all couplings of $\mu$ and $\hat\mu$. Equivalently, the Wasserstein distance may be defined by
\[
W_2^2(\mu,\hat \mu) = \inf \E\left[|Z-\hat Z|^2\right],
\]
where the infimum is taken over all joint distributions of the random variables $Z$ and $\hat Z$ with marginals $\mu$ and $\hat \mu$ respectively.
It is well-known that $W_2$ defines a metric on $\mathcal{P}_2(\R^d)$. Since $\R^d$ is a separable complete metric space and $p$ a positive number, the metric space $(\P_p(\R^d),W_p)$ is separable and complete,\cite{BolleyPolish} in particular this yields $(\P_2(\R^d),W_2)$ is Polish. We remind the reader that a Polish space is a separable completely metrizable topological space. With each point $\mu \in \P_2(\R^d)$ and every $\epsilon > 0$ we associate an $\epsilon$-ball $U_{\epsilon}(\mu) = \{ \nu \in \P_2(\R^d) \colon W_2(\mu,\nu) < \epsilon \}$.  The $\epsilon$-balls form the basis of a topology, called the topology of the metric space $(\P_2(\R^d),W_2)$. If this topology agrees with a given topology $\mathcal{O}$ on $\P_2(\R^d)$, we say that the metric $W_2$ is compatible with the topology $\mathcal{O}$.  If for a topological space $(X, \mathcal{O})$ there exists a compatible metric, we say $(X, \mathcal{O})$ is metrizable or that the metric metrizes the topological space $(X, \mathcal{O})$. For $\R^d$ and $p \in [1,\infty)$ it is well-known that the Wasserstein distance $W_p$ is compatible with the weak topology in $\P_p(\R^d)$.
Moreover, convergence in $W_p$ implies convergence of the first $p$ moments (see Chapter 6 of \cite{OldandNew} for more details). Altogether, $W_2$ metrizes the weak convergence in $\P_2(\R^d)$ and convergence in $W_2$ implies convergence of the first two moments. Moreover, $\P_2(\R^d)$ with the weak topology is Polish.

We split the results of this section into two parts, based on additional assumptions on $f$. We begin our investigation with the easier of the two, which also provides the means to prove the other. Throughout this section, we assume that $f$ satisfies the following assumptions:

\begin{assumption}\label{ass}
	\begin{enumerate}
    \item The cost function $f\colon\R^d\to\R$ is bounded from below with $\underline f := \inf f$.
    \item There exist constants $L_f$ and $c_u>0$ such that
    \begin{align}
    \left\lbrace\quad\begin{aligned}
    |f(x)-f(y)| &\le L_f(|x|+|y|)|x-y| \quad \text{for all\, $x,y\in\R^d$},\\
    f(x)-\underline f &\le c_u(1+|x|^2)\quad \text{for all\, $x\in\R^d$}.
    \end{aligned}\right.\tag{A1}\label{ass:1}
    \end{align}
	\end{enumerate}
\end{assumption}

\subsection{Bounded cost functions}\label{subsec:bounded} In addition to Assumption~\ref{ass} we consider cost functions $f$ that are bounded from above. In particular, $f$ has the upper bound $\overline f:=\sup f$.

\

The main result of this section is provided by the following theorem.

\begin{theorem}\label{thm:well_posed_bounded}
	Let $f$ satisfy Assumption~\ref{ass} and be bounded, and $\rho_0\in\P_4(\R^d)$.
	Then there exists a unique nonlinear process $\X\in \C([0,T],\R^d)$, $T>0$, satisfying 
	\begin{align*}
	d\X_t = -\lambda(\X_t - \m_f[\rho_t])\,dt + \sigma|\X_t - \m_f[\rho_t]|dW_t,\qquad \rho_t=\law(\X_t),
	\end{align*}
	in the strong sense, and $\rho\in \C([0,T],\P_2(\R^d))$ satisfies the corresponding Fokker--Planck equation \eqref{eq:fokker-planck} (in the weak sense) with $\lim\nolimits_{t\to 0}\rho_t=\rho_0\in\P_2(\R^d)$.
\end{theorem}

Before we prove the theorem, we discuss two results that not only facilitate the proof of Theorem~\ref{thm:well_posed_bounded}, but are also interesting in their own right. 

\begin{lemma}\label{lem:osc}
	Let $f$ satisfy Assumption~\ref{ass} and $\mu \in\P_2(\R^d)$ with $\int|x|^2d\mu\le K$. Then 
	\begin{align*}
	\frac{e^{-\alpha \underline f}}{\|\om\|_{L^1(\mu)}}\le \exp\bigl(\alpha c_u(1+ K)\bigr)=:c_K.
	\end{align*}
\end{lemma}
\begin{proof}
	The proof follows from the Jensen inequality, which gives
	\[
	e^{-\alpha \int (f-\underline f)\,d\mu} \le \int e^{-\alpha(f-\underline f)}d\mu.
	\]
	A simple rearrangement of the previous inequality and using \eqref{ass:1} yields the required estimate.
\end{proof}

 \begin{lemma}\label{lem:stability}
	Let $f$ satisfy Assumption~\ref{ass} and $\mu,\hat\mu\in \P_2(\R^d)$ with	
	\[
	\int|x|^4d\mu,\quad \int|\hat x|^4d\hat\mu \le K.
	\]
	Then the following stability estimate holds
	\begin{align*}
	|\m_f[\mu]-\m_f[\hat \mu]| \le c_0W_2(\mu,\hat{\mu}),
	\end{align*}
	for a constant $c_0>0$ depending only on $\alpha$, $L_f$ and $K$.
\end{lemma}
\begin{proof}
	Taking the difference, we obtain
	\begin{align*}
	\m_f[\mu]-\m_f[\hat \mu] = \iint \left[\frac{x\,\om(x)}{\|\om\|_{L^1(\mu)}} - \frac{\hat x\,\om(\hat x)}{\|\om\|_{L^1(\hat\mu)}}\right] d\pi =: \iint \big(h(x) - h(\hat x)\big)\, d\pi,
	\end{align*}
	where $\pi\in\Pi(\mu,\hat{\mu})$ is an arbitrary coupling of $\mu$ and $\hat\mu$. Observe that the integrand on the right may be written as
	\begin{align*}
	h(x)-h(\hat x) &= \frac{(x-\hat x)\,\om(\hat x)}{\|\om\|_{L^1(\mu)}} + \frac{x(\om(x)-\om(\hat x))}{\|\om\|_{L^1(\mu)}} +\frac{\iint \big(\om(\hat x) - \om(x)\big) d\pi_t}{\|\om\|_{L^1(\mu)}\|\om\|_{L^1(\hat\mu)}}\,\hat x\,\om(\hat x).
	\end{align*}
	Under the assumption \eqref{ass:1} on $f$ and Lemma~\ref{lem:osc}, the terms may be estimated by
	\begin{align*}
	|h(x)-h(\hat x)| &\le c_K|x-\hat x| + c_K \alpha L_f|x|(|x|+|\hat x|)|x-\hat x| \\
	&\hspace*{12em}+ c_K^2\alpha L_f|\hat x|\iint (|x|+|\hat x|)|x-\hat x|\,d\pi.
	\end{align*}
	Using this estimate, we further obtain
	\begin{align*}
	|\m_f[\mu]-\m_f[\hat \mu]| &\le c_K\iint (1 + \alpha L_f|x|(|x|+|\hat x|))|x-\hat x|\,d\pi \\
	&\hspace*{6em}+ c_K^2\alpha L_f\left(\int |\hat x|\,d\hat\mu\right)\iint (|x|+|\hat x|)|x-\hat x|\,d\pi \\
	&\le c_K\Bigl(1 + \alpha L_f(1+c_K) p_K \Bigr)\biggl(\iint |x-\hat x|^2d\pi\biggr)^{\frac{1}{2}},
	\end{align*}
	where we applied the H\"older inequality in the second line and $p_K$ is a polynomial in $K$. Finally,
	optimizing over all couplings $\pi$ concludes the proof.
\end{proof}

\begin{proof}[Proof of Theorem~\ref{thm:well_posed_bounded}]
	{\bf \em Step 1:} For some given $u\in\C([0,T],\R^d)$, we may uniquely solve the SDE
	\begin{align}\label{eq:linear_sde}
	dY_t = -\lambda(Y_t - u_t)dt + \sigma|Y_t-u_t|dW_t,\qquad \law(Y_0)=\rho_0,
	\end{align}
	for some fixed initial measure $\rho_0\in\P_4(\R^d)$, which induces $\nu_t=\law(Y_t)$. Since $Y\in\C([0,T],\R^d)$, we obtain $\nu\in\C([0,T],\P_2(\R^d))$, which satisfies the following Fokker--Planck equation
	\begin{align}\label{eq:linear_fp}
	\frac{d}{dt}\int \varphi\, d\nu_t = \int \Big((\sigma^2/2)|x-u_t|^2\Delta\varphi - \lambda(x-u_t)\cdot\nabla\varphi\Big)d\nu_t,
	\end{align}
	for all $\varphi\in\C_b^2(\R^d)$.
	Setting $\m_f[\nu]\in\C([0,T],\R^d)$ provides the self-mapping property of the map 
	\[
	\mathcal T\colon \C([0,T],\R^d)\to \C([0,T],\R^d);\;\; u\mapsto \mathcal T u=\m_f[\nu],
	\]
	for which we show to be compact.
	
	{\bf \em Step 2:} Since $\rho_0\in\P_4(\R^d)$, standard theory of SDEs (see e.g.~Chapter~7 of \cite{arnold1974stochastic}) provides a fourth-order moment estimate for solutions to \eqref{eq:linear_sde} of the form 
	\[
	 \E|Y_t|^4 \le (1+\E |Y_0|^4)e^{ct},
	\]
	for some constant $c>0$. In particular, $\sup_{t\in[0,T]}\int |x|^4d\nu_t\le K$ for some $K<\infty$. On the other hand, for any $t>s$, $t,s\in(0,T)$, the It\^o isometry yields
	\begin{align*}
	\E|Y_t-Y_s|^2 &\le 2\lambda^2 |t-s|\E\int_s^t |Y_\tau - u_\tau|^2 d\tau + 2\sigma^2 \E\int_s^t |Y_\tau-u_\tau|^2 d\tau \\
	&\le 4(\lambda^2T+\sigma^2)(K+\|u\|_\infty^2)|t-s|=: c\,|t-s|,
	\end{align*}
	and therefore, $W_2(\nu_t,\nu_s) \le c\,|t-s|^{\frac{1}{2}}$, for some constant $c>0$. Applying Lemma~\ref{lem:stability} with $\mu=\nu_t$ and $\hat{\mu}=\nu_s$, we obtain
	\[
	 |\m_f[\nu_t]-\m_f[\nu_s]| \le c_0W_2(\nu_t,\nu_s)\le c_0\,c\,|t-s|^{\frac{1}{2}},
	\]
	which provides the H\"older continuity of $t\mapsto\m_f[\nu_t]$ with exponent $1/2$, and thereby the compactness of $\mathcal T$ due to the compact embedding $\C^{0,1/2}([0,T],\R^d)\hookrightarrow\C([0,T],\R^d)$.
	
	{\bf \em Step 3:} Now let $u\in \C([0,T],\R^d)$ satisfy $u=\tau\mathcal Tu$ for $\tau\in[0,1]$. In particular, there exists $\rho\in\C([0,T],\P_2(\R^d))$ satisfying \eqref{eq:linear_fp} such that $u=\tau\m_f[\rho]$. Due to the boundedness assumption on $f$, we have for all $t\in(0,T)$ that
	\begin{align}\label{eq:uniform_bounded}
	 |u_t|^2=\tau^2|\m_f[\rho_t]|^2 \le \tau^2 e^{\alpha(\overline f-\underline f)}\int |x|^2 d\rho_t.
	\end{align}
	Therefore, a computation of the second moment gives
	\begin{align*}
	\frac{d}{dt}\int |x|^2 d\rho_t &= \int \Big(d\sigma^2|x-u_t|^2 - 2\lambda (x-u_t) \cdot x \Big)d\rho_t \\
	&= \int \Big((d\sigma^2-2\lambda)|x|^2 + 2\gamma\, x\cdot u_t + d\sigma^2|u_t|^2 \Big)d\rho_t \\
	&\le (d\sigma^2-2\lambda+|\gamma|)\int |x|^2d\rho_t + (d\sigma^2+|\gamma|) |u_t|^2 \le c_\lambda\int |x|^2d\rho_t,
	\end{align*}
	where $\gamma=\lambda-d\sigma^2$ and $c_\lambda=(d\sigma^2+|\gamma|)(1+e^{\alpha(\overline f-\underline f)})$. From Gronwall's inequality we easily deduce
	\[
	\int |x|^2 d\rho_t \le \left(\int |x|^2 d\rho_0\right)e^{c_\lambda t},
	\]
	and consequently also an estimate for $\|u\|_\infty$ via \eqref{eq:uniform_bounded}. In particular, there is a constant $q>0$ for which $\|u\|_\infty<q$. We conclude the proof by applying the Leray--Schauder fixed point theorem,\cite[Chapter~11]{gilbarg2015elliptic} which provides a fixed point $u$ for the mapping $\mathcal T$ and thereby a solution of \eqref{eq:nonlinear} (respectively \eqref{eq:fokker-planck}).
	
	 {\bf \em Step 4:} As for uniqueness, we first note that a fixed point $u$ of $\mathcal T$ satisfies $\|u\|_\infty<q$. Hence, the fourth-order moment estimate provided in {\bf\em Step 2} holds and $\sup_{t\in[0,T]}\int |x|^4d\rho_t\le K<\infty$. Now suppose we have two fixed points $u$ and $\hat u$ with 
	 \[
	  \|u\|_\infty, \|\hat u\|_\infty <q,\qquad \sup_{t\in[0,T]}\int |x|^4d\rho_t,\;\; \sup_{t\in[0,T]}\int |x|^4d\hat\rho_t\le K,
	 \]
	 and their corresponding processes $Y_t$, $\hat Y_t$ satisfying \eqref{eq:linear_sde} respectively. Then taking the difference $z_t:=Y_t-\hat Y_t$ for the same Brownian path gives
	 \[
	  z_t = z_0 -\lambda\int_0^t z_s\,ds + \lambda \int_0^t (u_s-\hat u_s)\,ds + \sigma \int_0^t \Bigl( |Y_s-u_s| - |\hat Y_s-\hat u_s|\Bigr) dW_s.
	 \]
	 Squaring on both sides, taking the expectation and applying the It\^o isometry yields
	 \[
	  \E|z_t|^2 \le 2\E|z_0|^2 + 8(\lambda^2 t + \sigma^2)\int_0^t \E|z_s|^2 ds + 4\lambda^2t\int_0^t |\m_f[\rho_t]-\m_f[\hat{\rho}_t]|^2 ds.
	 \]
	 Since Lemma~\ref{lem:stability} provides the estimate
	 \begin{align}\label{eq:unique}
	  |\m_f[\rho_t]-\m_f[\hat{\rho}_t]|\le c_0 W_2(\rho_t,\hat{\rho}_t) \le c_0 \sqrt{\E|z_t|^2},
	 \end{align}
	 we further obtain
	 \[
	  \E|z_t|^2 \le 2\E|z_0|^2 + 4\big((2 + c_0^2)\lambda^2 t + 2\sigma^2\big)\int_0^t \E|z_s|^2 ds.
	 \]
	 Therefore, applying Gronwall's inequality and using the fact that $\E|z_0|^2= 0$ yields $\E|z_t|^2=0$ for all $t\in[0,T]$. In particular, $\|u-\hat u\|_\infty=0$, i.e., $u\equiv \hat u$ due to \eqref{eq:unique}.
\end{proof}

\subsection{Cost functions with quadratic growth at infinity} In this subsection, we allow for cost functions that have quadratic growth at infinity. More precisely, we suppose the following:

There exist constants $M>0$ and $c_l>0$ such that
\begin{align}
	f(x)- \underline f \ge c_l|x|^2\quad\text{for\, $|x|\ge M$}\tag{A2}\label{ass:2}.
\end{align}
We provide a similar result as in the boundedness case under the assumption \eqref{ass:2}.

\begin{theorem}\label{thm:well_posed_growth}
	Let $f$ satisfy Assumption~\ref{ass} and quadratic growth \eqref{ass:2}, and $\rho_0\in\P_4(\R^d)$. Then there exists a unique nonlinear process $\X\in \C([0,T],\R^d)$, $T>0$, satisfying 
	\begin{align*}
	d\X_t = -\lambda(\X_t - \m_f[\rho_t])\,dt + \sigma|\X_t - \m_f[\rho_t]|dW_t,\qquad \rho_t=\law(\X_t),
	\end{align*}
	in the strong sense, and $\rho\in \C([0,T],\P_2(\R^d))$ satisfies the corresponding Fokker--Planck equation \eqref{eq:fokker-planck} (in the weak sense) with $\lim\nolimits_{t\to 0}\rho_t=\rho_0\in\P_2(\R^d)$.
\end{theorem}
\begin{proof}
 The proof is a slight modification of {\bf \em Step 3} in the proof of Theorem~\ref{thm:well_posed_bounded}. Since {\bf \em Steps 1}, {\bf \em 2} and {\bf \em 4} remain the same, we only show {\bf \em Step 3}. 
 
 {\bf \em Step 3:} Let $u\in \C([0,T],\R^d)$ satisfy $u=\tau\mathcal Tu$ for $\tau\in[0,1]$, i.e., there exists $\rho\in\C([0,T],\P_2(\R^d))$ satisfying \eqref{eq:linear_fp} such that $u=\tau\m_f[\rho]$. Due to Lemma~\ref{lem:moment} below, we have that
 \begin{align}\label{eq:uniform_growth}
 |u_t|^2=\tau^2|\m_f[\rho_t]|^2 \le \tau^2\int |x|^2 d\eta_t^{\alpha} \le \tau^2\biggl(b_1 + b_2\int |x|^2 d\rho_t\biggr),
 \end{align}
 for the constants $b_1$ and $b_2$ given in \eqref{eq:moment}. Therefore, a similar computation of the second moment estimate as above gives
 \begin{align*}
 \frac{d}{dt}\int |x|^2 d\rho_t &\le (d\sigma^2-2\lambda+|\gamma|)\int |x|^2 d\rho_t + \tau(d\sigma^2+|\gamma|) \biggl(b_1 + b_2 \int|x|^2 d\rho\biggr) \\
 &\le (d\sigma^2+|\gamma|)\, b_1 + (d\sigma^2+|\gamma|)(1+b_2)\int |x|^2 d\rho_t,
 \end{align*}
 which by Gronwall's inequality yields
 \[
   \int |x|^2 d\rho_t \le e^{c_\lambda t}\int |x|^2 d\rho_0 + \frac{b_1}{1+b_2}\big(e^{c_\lambda t}-1\big),
 \]
 with $c_\lambda=(d\sigma^2+|\gamma|)(1+b_2)$, and consequently also an estimate for $\|u\|_\infty$ via \eqref{eq:uniform_growth}. In particular, there is a constant $q>0$ for which $\|u\|_\infty<q$. We conclude the proof by using the same argument as in {\bf \em Step 3} in the proof of Theorem \ref{thm:well_posed_bounded}.
\end{proof}

\begin{lemma}\label{lem:moment}
 Let $f$ satisfy Assumption~\ref{ass} and \eqref{ass:2}, and $\mu\in \P_2(\R^d)$. Then
 \begin{align}\label{eq:moment}
  \int|x|^2d\eta^\alpha \le b_1 + b_2\int|x|^2 d\mu,\qquad \eta^\alpha=\om\,\mu/\|\om\|_{L^1(\mu)}
 \end{align}
 with constants
 \[
  b_1 = M^2 + b_2,\qquad b_2 = 2\frac{c_u}{c_l}\biggl(1 + \frac{1}{\alpha c_l}\frac{1}{M^2}\biggr),
 \]
 depending only on $M$, $c_u$ and $c_l$.
\end{lemma}
\begin{proof}
 We begin by looking at $\eta^\alpha(\{f-f_*\ge k\})$ for some $k>0$. A simple computation gives
 \[
  \eta^\alpha(\{f-f_*\ge k\}) = \frac{1}{\|\om\|_{L^1(\mu)}}\int_{\{f-f_*\ge k\}} e^{-\alpha f} d\mu \le \frac{e^{-\alpha(f_*+k)}}{\|\om\|_{L^1(\mu)}}\mu(\{f-f_*\ge k\}).
 \]
 On the other hand,
 \[
  \|\om\|_{L^1(\mu)} = \int e^{-\alpha f}\,d\mu \ge \mu(\{f-f_*< \ell\})e^{-\alpha(f_*+\ell)},
 \]
 for any $\ell>0$. Consequently, we obtain our first estimate
 \begin{align}\label{eq:estimate_1}
\eta^\alpha(\{f-f_*\ge k\}) \le e^{-\alpha(k-\ell)}\frac{\mu(\{f-f_*\ge k\})}{\mu(\{f-f_*< \ell\})},
 \end{align}
 which holds for any $k,\ell>0$, as long as the terms on the right-hand side are finite. 
 
 Now let us estimate the second moment of $\eta^\alpha$. For some $R_0\ge M$ to be determined later,
 \begin{align*}
 \int|x|^2 d\eta^\alpha 
 &\le R_0^2\eta^\alpha(B_{R_0}) + \int_{B_{R_0}^c}|x|^2 d\eta^\alpha = R_0^2 + \int_{B_{R_0}^c}\Big(|x|^2-R_0^2\Big) d\eta^\alpha \\
 &= R_0^2 + \sum_{i=1}^\infty\int_{A_i}\Big(|x|^2-R_0^2\Big) d\eta^\alpha \le R_0^2 + R_0^2\sum_{i=1}^\infty\big(2^{i}-1\big)\eta^\alpha(A_i),
 \end{align*}
 where $A_i:=\{ 2^{(i-1)/2}R_0 \le |x| < 2^{i/2} R_0 \}$, $i\in\N$. However, since \eqref{ass:2} tells us that $f-f_*\ge c_l|x|^2$ for all $|x|\ge R_0\ge M$, we find that $A_i\subset \{ f-f_*\ge c_l 2^{i-1}R_0^2 \}$ and therefore $\eta^\alpha(A_i)\le \eta^\alpha(\{ f-f_*\ge c_l 2^{i-1}R_0^2 \})$ for all $i\in\N$. From estimate \eqref{eq:estimate_1} and the choice $\ell = c_l 2^{i-2}R_0^2$, we further obtain
 \begin{align}\label{eq:estimate_2}
 \begin{aligned}
 \int|x|^2 d\eta^\alpha &\le R_0^2 + R_0^2\sum_{i=1}^\infty\big(2^{i}-1\big)\eta^\alpha(\{ f-f_*\ge c_l 2^{i-1}R_0^2 \}) \\
 &\le R_0^2 + R_0^2\sum_{i=1}^\infty\big(2^{i}-1\big)e^{-\alpha c_l R_0^2 2^{i-2}}\frac{\mu(\{ f-f_*\ge c_l 2^{i-1}R_0^2 \})}{\mu(\{ f-f_* < c_l 2^{i-2}R_0^2 \})}.
 \end{aligned}
 \end{align}
We are now left with estimating the last term on the right.
 
 The Markov inequality and \eqref{ass:1} yields
 \[
 \mu(\{ f-f_*\ge k \}) \le \frac{1}{k} \int|f-f_*|\,d\mu \le \frac{1}{k}c_u\biggl(1 + \int|x|^2 d\mu\biggr) =: \frac{c_u}{k}m,
 \]
 for any $k>0$. Furthermore, since
 \[
 \mu(\{f-f_*< k\}) = 1 - \mu(\{f-f_*\ge k\}) \ge 1 - \frac{c_u}{k}m,
 \]
 we have that
 \begin{align*}
 \frac{\mu(\{ f-f_*\ge k \})}{\mu(\{ f-f_* < k/2 \})} \le \frac{1}{2}\frac{m}{k/(2c_u) - m}.
 \end{align*}
 Choosing $k= c_l 2^{i-1}R_0^2$, we obtain
 \begin{align*}
 \frac{\mu(\{ f-f_*\ge c_l 2^{i-1}R_0^2 \})}{\mu(\{ f-f_* < c_l 2^{i-2}R_0^2 \})} \le \frac{1}{2}\frac{m}{(c_l/c_u) 2^{i-2}R_0^2 - m}.
 \end{align*}
 At this point, we can choose
 \[
 R_0^2 = M^2 + 2\frac{c_u}{c_l}m>M^2,
 \]
 to further obtain
 \begin{align*}
 \frac{\mu(\{ f-f_*\ge c_l 2^{i-1}R_0^2 \})}{\mu(\{ f-f_* < c_l 2^{i-2}R_0^2 \})} \le \frac{1}{2}\frac{m}{(c_l/c_u) 2^{i-2}M^2 + (2^{i-1}-1)m} \le 2\frac{c_u}{c_l}\frac{m}{2^iM^2}.
 \end{align*}
 Inserting this into \eqref{eq:estimate_2} yields
 \begin{align*}
 \int|x|^2 d\eta^\alpha &\le R_0^2 + 2\frac{c_u}{c_l}\frac{m}{M^2} R_0^2\sum_{i=1}^\infty\big(1-2^{-i}\big)e^{-\alpha c_l R_0^2 2^{i-2}}.
 \end{align*}
 A coarse estimate of the summation on the right may be given by
 \begin{align*}
 \sum_{i=1}^\infty\big(1-2^{-i}\big)e^{-\alpha c_l R_0^2 2^{i-2}} &\le \sum_{i=1}^\infty e^{-\alpha c_l R_0^2 2^{i-2}} \le \int_{\frac12}^\infty e^{-\alpha c_l R_0^2 s}\,ds 
 \le \frac{1}{\alpha c_l R_0^2},
 \end{align*}
 which finally gives an estimate of the form
 \[
 \int|x|^2 d\eta^\alpha \le M^2 + 2\frac{c_u}{c_l}\biggl(1 + \frac{1}{\alpha c_l}\frac{1}{M^2}\biggr) \biggl(1 + \int|x|^2 d\mu\biggr) =: b_1 + b_2 \int|x|^2 d\mu,
 \]
 thereby concluding the proof.
\end{proof}

\begin{remark}
	Observe that the estimate \eqref{eq:moment} provided in Lemma~\ref{lem:moment} does not blow up as $\alpha\to\infty$, in contrast to estimate \eqref{eq:uniform_bounded} used in the proof of the bounded case. In fact, the constant $b_2$ may be chosen to be independent of $\alpha$ for $\alpha\ge 1$. 
\end{remark}

\subsection{Moment estimates for the stochastic empirical measure}
For the solution ${\bf X}^{(N)}\in \C([0,T],\R^d)^N$, $N\in\N$ of the particle system \eqref{eq:particle}, we denote by 
\[
\rho_t^N = \frac{1}{N}\sum\nolimits_{i=1}^N \delta_{X_t^{(i,N)}}\in\P_2(\R^d),
\]
the empirical measure corresponding to ${\bf X}_t^{(N)}$ for every $t\in[0,T]$, where $\delta_x$ is the Dirac measure at $x\in\R^d$. Note that since ${\bf X}_t^{(N)}$ is a random variable, so is its empirical measure $\rho_t^N$.

\begin{lemma}\label{lem:moment_particle}
	Let $f$ satisfy Assumption~\ref{ass} and either (1) boundedness, or (2) quadratic growth at infinity \eqref{ass:2}, and $\rho_0\in \P_{2p}(\R^d)$, $p\ge 1$. Further, let ${\bf X}^{(N)}$, $N\in\N$ be the solution of the particle system \eqref{eq:particle} with $\rho_0^{\otimes N}$-distributed initial data ${\bf X}_0^{(N)}$ and $\rho^N$ the corresponding empirical measure. Then there exists a constant $K>0$, independent of $N$, such that
	\begin{align}\label{eq:uniform_particle}
	\sup\nolimits_{t\in[0,T]}\E\int |x|^{2p} d\rho_t^N,\qquad \sup\nolimits_{t\in[0,T]}\E|m_t^N|^{2p} \le K,
	\end{align}
	and consequently also the estimates
	\[
	 \sup\nolimits_{t\in[0,T]}\E\int |x|^{2} d\eta_t^{\alpha,N},\qquad \sup\nolimits_{t\in[0,T]} \E|X_t^{(i,N)}|^{2p} \le K\qquad \text{for\, $i=1,\ldots,N$},
	\]
	for the same constant $K$.
\end{lemma}
\begin{proof}
	Let ${\bf X}^{(N)}$, $N\in\N$ be a solution of the particle system \eqref{eq:particle}. Using the inequality $(a+b)^q\le 2^{q-1}(a^q+b^q)$, $q\ge 1$ and the It\^o isometry, it is easy to see that the estimate
	\begin{align}\label{eq:moment_particle_i}
		\begin{aligned}
			\E|X_t^{(i,N)}|^{2p} &= 2^{2p-1}\E|X_0^{(i,N)}|^{2p} \\
			&+ 2^{3(2p-1)}(\lambda^{2p}T^{p}+\sigma^{2p})T^{p-1}\int_0^t \E|X_s^{(i,N)}|^{2p}+\E|m_s^N|^{2p} ds,
		\end{aligned}
	\end{align}
	holds for each $i=1,\ldots,N$. Summing the previous inequality over $i=1,\ldots, N$, dividing by $N$ and using the linearity of the expectation gives
	\begin{align}\label{eq:moment_particle}
	\begin{aligned}
	\E\int|x|^{2p}d\rho_t^N &\le 2^{2p-1}\E\int|x|^{2p}d\rho_0^N \\
	&\hspace*{-1em}+ 2^{3(2p-1)}(\lambda^{2p}T^{p}+\sigma^{2p})T^{p-1}\int_0^t \biggl[\E\int|x|^{2p}d\rho_s^N +\E|\m_s^N|^{2p}\biggr]ds.
	\end{aligned}
	\end{align}
	Following the strategy given in Section~\ref{sec:mesoscopic}, we obtain for cost functions $f$ satisfying Assumption~\ref{ass} and either boundedness or \eqref{ass:2}, the estimate
	\begin{align}\label{eq:moment_particle_m}
	 |\m_s^N|^2 \le \int|x|^2d\eta_s^{\alpha,N} \le c_1 + c_2\, \int |x|^2 d\rho_s^N,
	\end{align}
	for appropriate constants $c_1$ and $c_2$, independent of $N$, where by construction $\m_t^N = \int x\,d\eta_t^{\alpha,N}$ with $\eta_t^{\alpha,N} = \om \rho_t^N/\|\om\|_{L^1(\rho_t^N)}\in\P(\R^d)$. Therefore, we further obtain
	\[
	|\m_s^N|^{2p} \le \biggl(c_1 + c_2\, \int |x|^2 d\rho_s^N\biggr)^p \le 2^{p-1}\biggl(c_1^p + c_2^p\, \int |x|^{2p} d\rho_s^N\biggr).
	\]
	Inserting this into \eqref{eq:moment_particle} and applying the Gronwall inequality provides a constant $K_p>0$, independent of $N$, such that $\sup_{t\in[0,T]}\E\int |x|^{2p} d\rho_t^N\le K_p$ holds, and consequently also
	\[
	 \sup\nolimits_{t\in[0,T]}\E|m_t^N|^{2p}\le 2^{p-1}(c_1^p + c_2^p\, K_p),
	\]
	which concludes the proof of the estimates in \eqref{eq:uniform_particle} by choosing $K$ sufficiently large. The other two estimates easily follow by \eqref{eq:moment_particle_m} and by applying the Grownwall inequality on \eqref{eq:moment_particle_i}, respectively.
\end{proof}

\begin{remark}
 Despite having the estimates provided in Lemma~\ref{lem:moment_particle}, we are unable to prove a mean-field limit result for the interacting particle system \eqref{eq:particle} towards the nonlinear process \eqref{eq:nonlinear} by means of classical tools (see e.g.~\cite{sznitman1991topics}). At this moment, we only postulate its corresponding mean-field equation \eqref{eq:fokker-planck}, for which the numerical simulations have indicated to be true.
\end{remark}


\section{Large Time Behavior and Consensus Formation}\label{sec_lt}

We finally arrive at the most important part of the paper, i.e., we provide sufficient conditions on $f$ such that {\em uniform consensus formation} happens. More precisely, we say that uniform consensus occurs when 
\[
\rho_t \longrightarrow \delta_{\tilde x}\quad\text{as}\;\; t\to\infty,
\]
for some $\tilde x\in\R^d$ possibly depending on $\rho_0$. 

In fact, in the framework of global optimization, we would like to further have that $\tilde x = x_*=\arg\min f$. In other words, we want that $\rho_t$ concentrates at the global minimum of $f$. Unlike, the deterministic case, the formation of nonuniform consensus, i.e., multiple opinions in the limit $t\to\infty$, in the stochastic model cannot occur.\cite{doi:10.1142/S0218202517400061} Hence, it is expected that uniform consensus is formed, whenever concentration happens. We will see that this is the case.

\begin{assumption}\label{ass:3}
Throughout this section we will assume that $f \in \mathcal C^2(\R^d)$ satisfies additionally
\begin{enumerate}
	\item $\inf f > 0$.
	\item $\|\nabla^2 f\|_\infty\le c_f$ and there exists constants $c_0,c_1>0$, such that
	\[
	\Delta f \le c_0 + c_1|\nabla f|^2\qquad\text{in\; $\R^d$}.
	\]
\end{enumerate}
\end{assumption}

%

Let us point out that from the viewpoint of global optimization we can change the value of the function $f$ outside a large ball in a way that condition (2) of Assumption~\ref{ass:3} is satisfied as soon as it does not change the point where the global minimum is achieved. These cases include interesting functions in applications. 

\subsection{Concentration estimate}
Let us begin by denoting the expectation and variance of $\rho_t$ by
\[
E(\rho_t) := \int x\,d\rho_t \quad \mbox{and} \quad V(\rho_t):= \frac{1}{2}\int |x- E(\rho_t)|^2\,d\rho_t.
\]
A simple computation of the evolution of its variance gives
\begin{equation*}
\frac{d}{dt} V(\rho_t)=\frac{d}{dt} \frac{1}{2}\int |x- E(\rho_t)|^2\,d\rho_t = -2\lambda V(\rho_t) + (d\sigma^2/2)\int|x-\m_f[\rho_t]|^2 d\rho_t.
\end{equation*}
To estimate the last term on the right, we apply Jensen's inequality to obtain
\begin{align}\label{eq:var22}
\int |x-\m_f[\rho_t]|^2 d\rho_t \le \frac{\iint |x-y|^2 \om d\rho_t(x)d\rho_t(y)}{\|\om\|_{L^1(\rho_t)}} \le 2 \frac{e^{-\alpha \underline f}}{\|\omega^\alpha_f\|_{L^1(\rho_t)}}V(\rho_t),
\end{align}
and therefore
\begin{align}\label{eq:estimate_var}
 \frac{d}{dt} V(\rho_t) \le -\Big(2\lambda - d\sigma^2 e^{-\alpha \underline f}/\|\omega^\alpha_f\|_{L^1(\rho_t)}\Big)V(\rho_t).
\end{align}
\begin{remark}
 If $f$ were bounded, then an obvious way to obtain concentration is to use the estimate $\|\omega^\alpha_f\|_{L^1(\rho_t)}\ge e^{-\alpha \overline f}$ to obtain
 \[
  \frac{d}{dt} V(\rho_t) \le -\Big(2\lambda - d\sigma^2 e^{\alpha (\overline f-\underline f)}\Big)V(\rho_t).
 \]
 However, since $\overline f >\underline f$, we have that $e^{\alpha(\overline f-\underline f)}\to\infty$ as $\alpha\to\infty$, and we would have to choose $\lambda\gg 1$ sufficiently large in order to obtain concentration. In the next subsection, we will see that this is not desirable since $\alpha$ has to be chosen large to have a good approximation of the global minimizer.
\end{remark}

In order to understand how $\|\om\|_{L^1(\rho_t)}$ evolves, we study its evolution given by
\begin{align*}
 \frac{d}{dt}\|\omega^\alpha_f\|_{L^1(\rho_t)} &= \alpha \lambda \int (\nabla f(x)-\nabla f(m_f[\rho_t])) \cdot (x - m_f[\rho_t])\, \omega^\alpha_fd\rho_t \\
 &+\frac{\sigma^2}{2}\int \lt(\alpha^2|\nabla f(x)|^2 - \alpha \Delta f(x) \rt)|x - m_f[\rho_t]|^2\omega^\alpha_f d\rho_t =: I_1 + I_2
\end{align*}
where we used the fact that
\[
\int (x - \m_f[\rho_t]) \,\omega^\alpha_f d\rho_t = 0,
\]
for the first term. 
\begin{lemma}
 Under Assumption~\ref{ass:3} for $f$ and $\alpha\ge c_1$, we have 
 \begin{align}\label{eq:estimate_om}
  \frac{d}{dt}\|\omega^\alpha_f\|_{L^1(\rho_t)}^2 \ge - b_1 V(\rho_t),
  \end{align}
  with $b_1=b_1(\alpha,\lambda,\sigma)=2\alpha e^{-2\alpha \underline f}( c_0\sigma^2 + 2\lambda c_f)$. 
  
  Note that from Assumption~\ref{ass:3}, we see that $b_1\to 0$ as $\alpha\to\infty$.
\end{lemma}
\begin{proof}
 From the assumptions on $f$, we obtain the following estimates
 \begin{align*}
 I_1 &\ge - \alpha\lambda c_f e^{-\alpha \underline f} \int |x - m_f[\rho_t]|^2 d\rho_t\\
 I_2 &\ge \alpha(\alpha-c_1)\frac{\sigma^2}{2}\int |\nabla f(x)|^2|x - m_f[\rho_t]|^2 \omega^\alpha_f\,d\rho_t \\
 &\hspace*{12em}- \alpha c_0\frac{\sigma^2}{2}e^{-\alpha \underline f}\int |x - m_f[\rho_t]|^2 d\rho_t.
 \end{align*}
 When $\alpha\ge c_1$, the first term in the estimate for $I_2$ is non-negative. Hence, we obtain
 \begin{align*}
 \frac{d}{dt}\|\omega^\alpha_f\|_{L^1(\rho_t)} &\ge -\alpha e^{-\alpha \underline f}( c_0\sigma^2/2 + \lambda c_f)\int |x - m_f[\rho_t]|^2 d\rho_t \\
 &\ge -\alpha e^{-2\alpha \underline f}( c_0\sigma^2 + 2\lambda c_f)\frac{V(\rho_t)}{\|\omega^\alpha_f\|_{L^1(\rho_t)}},
 \end{align*}
 where the second inequality follows from \eqref{eq:var22}. We obtain \eqref{eq:estimate_om} from the fact that 
 \[
  \frac{1}{2}\frac{d}{dt}\|\omega^\alpha_f\|_{L^1(\rho_t)}^2=\|\omega^\alpha_f\|_{L^1(\rho_t)}\frac{d}{dt}\|\omega^\alpha_f\|_{L^1(\rho_t)} \ge -\alpha e^{-2\alpha \underline f}( c_0\sigma^2 + 2\lambda c_f)V(\rho_t),
 \]
 thereby concluding the proof.
\end{proof}

We now have the ingredients to show the concentration of $\rho_t$. In particular, we show that the estimates \eqref{eq:estimate_var} and \eqref{eq:estimate_om} provide the means to identify assumptions on the parameters $\alpha,\lambda$ and $\sigma$, for which we obtain the convergence $V(\rho_t)\to0$ as $t\to\infty$ at an exponential rate.

\begin{theorem}\label{thm:concentration}
 Let $f$ satisfy Assumption~\ref{ass:3}. Furthermore let the parameters $\alpha,\lambda$ and $\sigma$ satisfy
 \[
  b_1 < \frac{3}{4},\qquad 2\lambda b_0^2 - K - 2d\sigma^2 b_0e^{-\alpha \underline f} \ge 0, 
 \]
 with $K=V(\rho_0)$ and $b_0=\|\omega^\alpha_f\|_{L^1(\rho_0)}$. Then $V(\rho_t)\le V(\rho_0)e^{-q t}$ with
 \[
  q = 2\big(\lambda - (d\sigma^2/b_0)e^{-\alpha \underline f}\big) \ge K/b_0^2.
 \]
 In particular, there exists a point $\tilde x\in\R^d$ for which $$E(\rho_t)\to \tilde x \quad\text{and}\quad \m_f[\rho_t]\to \tilde x \quad\text{as}\quad t\to\infty.$$
\end{theorem}
\begin{remark}
 Since $b_1\to 0$ and $e^{-\alpha\underline f}\to 0$ as $\alpha\to\infty$, the set of parameters for which the above conditions are satisfied is non-empty.
\end{remark}

\begin{proof}[Proof of Theorem~\ref{thm:concentration}] 
 Let $\gamma=K/b_0^2$. Choose any $\ve >0$ and set
 \[
 \mathcal{T}^\ve := \Bigl\{ t > 0 \,:\, V(\rho_s)e^{\gamma s} < K_\ve  \quad \mbox{for} \quad s \in [0,t) \Bigr\},
 \]
 where $K_\ve := K + \ve$. Then, by continuity of $V(\rho_t)$, we get $\mathcal{T}^\ve \neq \emptyset$. Set $\mathcal{T}^\ve_* := \sup \mathcal{T}^\ve$, and we claim $\mathcal{T}^\ve_* = \infty$. Suppose not, i.e., $\mathcal{T}^\ve_* < \infty$. Then this yields
 \[
 \lim\nolimits_{t \nearrow \mathcal{T}^\ve_*} V(\rho_t)e^{\gamma t} = K_\ve .
 \]
 On the other hand, it follows from \eqref{eq:estimate_om} that for $t < \mathcal{T}^\ve_*$, we have
 \begin{align*}
 \|\om\|_{L^1(\rho_t)}^2 &\geq b_0^2 - b_1\int_0^t V(\rho_s)\,ds > b_0^2 - b_1K_\ve\int_0^t e^{-\gamma s}\,ds \\
 &= b_0^2\lt(1 - b_1\frac{K_\ve}{K}\rt) + \frac{b_1K_\ve}{\gamma}e^{-\gamma \mathcal{T}^\ve_*}.
 \end{align*}
 Since $b_1 < 3/4$ and $K_\ve \to K$ as $\ve \to 0$, there exists $\ve_0 > 0$ such that $1 - b_1 K_\ve/K \geq 1/4$ for $0 < \ve \leq \e_0$. Thus we obtain that for $t < \mathcal{T}^\ve_*$ with $0 < \ve \leq \ve_0$
 \[
 \|\om\|_{L^1(\rho_t)} \geq \frac{b_0}{2} + \xi_\ve \quad \mbox{for some} \quad \xi_\ve > 0.
 \]
 Inserting this into \eqref{eq:estimate_var} gives
 \begin{align*}
 \frac{d}{dt} V(\rho_t) \le -2\lt(\lambda - \frac{d\sigma^2}{b_0 + 2\xi_\e} e^{-\alpha \underline f}\rt)V(\rho_t).
 \end{align*}
 Applying Gronwall's inequality yields
 \[
 V(\rho_t)e^{q_\ve t} \le K  \quad \mbox{where} \quad q_\ve = 2\lt(\lambda - \frac{d\sigma^2}{b_0 + 2\xi_\ve} e^{-\alpha \underline f}\rt),
 \]
 for $t < \mathcal{T}^\ve_*$ with $0 < \ve \leq \ve_0$. On the other hand, taking the limit $t \nearrow \mathcal{T}^\ve_*$ to the above inequality gives
\[
K_\ve  =  \lim\nolimits_{t \to \mathcal{T}^\ve_*-} V(\rho_t)e^{\gamma t} <  \lim\nolimits_{t \to \mathcal{T}^\ve_*-} V(\rho_t)e^{q_\ve t} =  K < K_\e,
\]
where we used the second condition, which gives $q_\ve > \gamma$ for all $0 < \ve \leq \ve_0$. This is a contradiction and implies that $\mathcal{T}^\ve_* = \infty$. Finally, by taking the limit $\ve \to 0$, we complete the proof.

%
 For the second part of the statement, we first observe that the expectation of $\rho_t$ satisfies
 \begin{align}\label{eq:mean1}
 \frac{d}{dt}E(\rho_t) = \frac{d}{dt}\int x\,d\rho_t = -\lambda \int (x-\m_f[\rho_t])\,d\rho_t.
 \end{align}
 Taking the absolute value of the equation above and then integrating in time yields 
 \begin{align*}
 \int_0^t \left|\frac{d}{ds} E(\rho_s)\right| ds &\leq \lambda\int_0^t \int |x-\m_f[\rho_s]|\,d\rho_s\,ds \le 4 (\lambda/b_0) e^{-\alpha\underline f} V(\rho_0)\int_0^t e^{-qs} ds\\
 &= \frac{4\lambda}{q b_0}e^{-\alpha\underline f} V(\rho_0)(1-e^{-q t}) \le \frac{4\lambda}{q b_0}e^{-\alpha\underline f} V(\rho_0),
 \end{align*}
 where we used the fact that $\|\om\|_{L^1(\rho_t)}^2\ge b_0/2$ and the concentration estimate $V(\rho_t)\le V(\rho_0)e^{-qt}$. The previous estimate tells us that $d E(\rho_t)/dt\in L^1(0,\infty)$ and thus, provides the existence of a point $\tilde x\in\R^d$, possibly depending on $\rho_0$, such that
 \[
 \hat  x = E(\rho_0) + \int_0^\infty \frac{d}{dt} E(\rho_t)\,dt = \lim\nolimits_{t\to\infty} E(\rho_t).
 \]
Moreover, since $|E(\rho_t)-\m_f[\rho_t]|\to 0$ for $t\to\infty$, we obtain $\lim_{t\to\infty} \m_f[\rho_t]=\tilde x$.
\end{proof}

\begin{remark}
 A very important takeaway from the conditions provided in Theorem~\ref{thm:concentration} is the fact that $\lambda$ and $\sigma$ may be chosen independently of $\alpha$. Indeed, if 
 \begin{align}\label{eq:condition_strong}
 	2\lambda b_0^2 - K - 2d\sigma^2 b_0 \ge 0,
 \end{align}
 then the second condition of Theorem~\ref{thm:concentration} is trivially satisfied since $f\ge0$ by assumption. Therefore, when \eqref{eq:condition_strong} is satisfied, consensus is achieved for arbitrarily large $\alpha$ satisfying $b_1\le 3/4$.
\end{remark}

\subsection{Approximate global minimizer}

While the previous results provided a sufficient condition for uniform consensus to occur, we will argue further that the point of consensus $\tilde x\in\R^d$ may be made arbitrarily close to the global minimizer $x_*$ of $f$, for $f\in \mathcal{C}^2(\R^d)$ satisfying Assumption~\ref{ass:3}. 

In the following, we assume further that $f$ attains a unique global minimum at $x_*\in \text{supp}(\rho_0)$.

\begin{theorem}\label{thm:approximate}
 Let $f$ satisfy Assumption~\ref{ass:3}. For any given $0<\e_0\ll1$ arbitrarily small, there exist some $\alpha_0\gg 1$ and appropriate parameters $(\lambda,\sigma)$ such that uniform consensus is obtained at a point $\tilde x\in B_{\e_0}(x_*)$. More precisely, we have that $\rho_t \to \delta_{\tilde x}$ for $t\to\infty$, with $\tilde x\in B_{\e_0}(x_*)$.
\end{theorem}

For the proof of Theorem~\ref{thm:approximate}, we will make use of the following auxiliary result.

\begin{lemma}\label{lem:approximate}
	Let $f$ satisfy Assumption~\ref{ass:3} and $\rho_t$ satisfy
	\[
	(1 - b(\alpha))\|\omega^\alpha_f\|_{L^1(\rho_0)}^2 \le \|\omega^\alpha_f\|_{L^1(\rho_t)}^2 \qquad\text{for all\; $t\ge 0$},
	\]
	with $0 \leq b(\alpha) \leq 1$ and $b(\alpha) \to 0$ as $\alpha \to \infty$.  Furthermore, assume that $V(\rho_t)\to0$ and $E(\rho_t)\to \tilde x$ as $t\to\infty$. Then for any given $0<\e_0\ll1$ arbitrarily small, there exist some $\alpha_0\gg 1$ such that $\tilde x\in B_{\e_0}(x_*)$.
\end{lemma}
\begin{proof}
	From the assumptions $V(\rho_t)\to 0$ and $E(\rho_t)\to \tilde x$ as $t \to \infty$, we deduce
	\begin{align*}
	\biggl|\int e^{-\alpha f} d\rho_t - \int_{B_\delta(\tilde x)} e^{-\alpha f} d\rho_t\biggr| &= \int_{\{|x-\tilde x|\ge k\}} e^{-\alpha f} d\rho_t  \le \rho_t(\{|x-\tilde x|\ge k\}) \\
	&\le \frac{1}{k^2}\int_{\{|x-\tilde x|\ge k\}} |x-\tilde x|^2\,d\rho_t \\
	&\le \frac{2}{k^2}\Bigl(V(\rho_t) + |E(\rho_t)-\tilde x|^2\Bigr),
	\end{align*}
	for any $k>0$. Furthermore, from Chebyshev's inequality, we know that $\rho_t \rightharpoonup \delta_{\tilde x}$ narrowly as $t\to\infty$. Thus, since $f$ is continuous and bounded on $B_{\delta}(\tilde x)$, we may pass to the limit in $t$ to obtain
	\[
	\lim_{t\to\infty} \int_{B_{\delta}(\tilde x)} e^{-\alpha f} d\rho_t = e^{-\alpha f(\tilde x)}.
	\]
	Altogether, we obtain the convergence
	\begin{equation}
	\|\omega^\alpha_f\|_{L^1(\rho_t)} \to e^{-\alpha f(\tilde x)} \quad \mbox{as} \quad t \to \infty.
	\end{equation}
	On the other hand, the Laplace principle (see e.g.~\cite{doi:10.1142/S0218202517400061}) gives us the possibility to choose $\alpha_1 \gg1$ large enough for any given $\e > 0$, such that
	\[
	-\frac1\alpha \log \|\omega^\alpha_f\|_{L^1(\rho_0)} - f_* < \frac\e2
	\]
	for any $\alpha\ge \alpha_1$. By assumption, we have that
	\[
	-\frac1\alpha \log \|\omega^\alpha_f\|_{L^1(\rho_t)} \leq -\frac1\alpha \log \|\omega^\alpha_f\|_{L^1(\rho_0)} - \frac{1}{2\alpha}\log(1 - b(\alpha)).
	\]
	From the continuity of the logarithm, we can pass to the limit $t\to\infty$ to obtain
	\[
	f(\tilde x) \leq -\frac1\alpha \log \|\omega^\alpha_f\|_{L^1(\rho_0)} - \frac{1}{2\alpha}\log(1 - b(\alpha)).
	\]
	This, together with using the assumption that $b(\alpha) \to 0$ as $\alpha \to \infty$, yields 
	\[
	0\le f(\tilde x) - f_*  < \frac\e2 - \frac{1}{2\alpha}\log(1 - b(\alpha)) < \e.
	\]
	for $\alpha \ge \alpha_2$ with $\alpha_2\gg 1$ sufficiently large.
	Finally, using the continuity of $f$ and the uniqueness of the global minimum, we can find a suitable small $\e_0>0$ (correspondingly $\alpha_0\ge \max\{\alpha_1,\alpha_2\}$) such that $\tilde x\in B_{\e_0}(x_*)$ for all $\alpha\ge \alpha_0$.
\end{proof}

\begin{proof}[Proof of Theorem~\ref{thm:approximate}]
	It follows from Lemma 4.1 that 
	\[
	\|\omega^\alpha_f\|_{L^1(\rho_t)}^2 \geq \|\omega^\alpha_f\|_{L^1(\rho_0)}^2 - b_1(\alpha)\int_0^t V(\rho_s)\,ds,
	\]
	where $b_1(\alpha) \to 0$ as $\alpha \to \infty$. Furthermore, by choosing $\lambda$ and $\sigma$ according to \eqref{eq:condition_strong}, we obtain from Theorem \ref{thm:concentration} that
	\[
	 \int_0^t V(\rho_s)\,ds \le \int_0^t V(\rho_0)e^{-qs}\,ds = \frac{V(\rho_0)}{q}(1-e^{-qt}) \le (1-e^{-qt})\|\omega^\alpha_f\|_{L^1(\rho_0)}^2,
	\]
	and consequently,
	\[
	\|\omega^\alpha_f\|_{L^1(\rho_t)}^2 \geq \|\omega^\alpha_f\|_{L^1(\rho_0)}^2(1 - b_1(\alpha)). 
	\]
	We conclude the proof by choosing $\alpha\ge \alpha_0$ with $\alpha_0$ obtained from Lemma~\ref{lem:approximate}.
\end{proof}


\section{Pseudo-inverse Distribution, Extended Models and Numerical Results}

In this section, we consider the Fokker--Planck equation \eqref{eq:fokker-planck} in one dimension and derive an equivalent formulation of the equation in terms of the pseudo-inverse distribution function. Then, we introduce an extension of the current model to replace the diffusion term with nonlinear diffusions of porous media type, which would guarantee compact support of the probability measure $\rho_t$. In this case the pseudo-inverse distribution function allows us to investigate concentration by only considering the evolution of the boundary points. To investigate this numerically, we introduce schemes for the porous media type equation and the evolution equation for the pseudo-inverse distribution function. We conclude with numerical results that illustrate the convergence results shown above.

\subsection{Evolution of the inverse distribution function}
We first define the well-known cumulative distribution $F_t$ of a probability measure $\rho_t$ and its pseudo-inverse $\chi_t$ (see e.g.~\cite{villani2003topics}) by
\begin{align*}
  F_t(x) = \int_{-\infty}^x d\rho_t = \rho_t((-\infty,x]).
 \end{align*}
Then, the pseudo-inverse of $F_t$ on the interval $[0,1]$ is defined by
 \begin{align*}
  \chi_t(\eta) := F_t^{-1}(\eta) := \inf\{ x \in \R\;|\; F_t(x) > \eta \}.
 \end{align*}
Both, $F_t$ and $\chi_t$ are by definition right-continuous. To derive the evolution equation for $\chi_t$ we use the properties of $\rho_t$, $F_t$ and $\chi_t$ collected in the following corollary. These properties can be obtained by basic calculus using the above definitions.

\begin{corollary}\label{cor:propChi}
 Let $\rho_t$ be a probability measure, $F_t$ the corresponding cumulative distribution and $\chi_t$ the pseudo-inverse distribution of $F_t$ as defined above. Then, the following equations hold
 \begin{gather*}
  \chi(t, F(t,x)) = x, \qquad F(t,\chi(t,\eta)) = \eta, \qquad \partial_t F = -\rho \partial_t \chi, \qquad\partial_\eta \chi = \frac{1}{\rho},\\
  \partial_\eta \chi\, \partial_x F = 1, \qquad \frac{\partial_x \rho}{\rho} = \frac{\partial_{xx}F}{\partial_x F} = -\frac{\partial_{\eta \eta} \chi}{(\partial_\eta \chi)^2}
 \end{gather*}
restricted to $x = \chi(t,\eta)$, $\eta = F(t,x)$ and $x \in supp(\rho)$, respectively.
\end{corollary}
From these properties we may now derive an integro-differential equation for the pseudo-inverse $\chi_t$, $t\ge 0$. Indeed, let us consider the solution $\rho_t$ to \eqref{eq:fokker-planck} satisfying $\rho_t \in \C([0,\infty), \P_2(\R))$. 
From the definition of $F_t$, we deduce that
\[
 \partial_t F_t(x) - \mu_t(x) \rho_t(x) = \rho_t(x) \partial_x \kappa_t(x) + \kappa_t(x) \partial_x \rho_t(x).
\]
We then use the relations between $\chi_t$ and $\rho_t$ provided in Corollary~\ref{cor:propChi} to obtain
\[
 -\rho_t(x) \partial_t \chi_t(\eta) - \mu_t(x) \rho_t(x) = \rho_t(x) \partial_x \kappa_t(x) + \kappa_t(x) \partial_x \rho_t(x),
\]
which consequently yields
\[
 \partial_t \chi_t(\eta) + \mu_t(x) = - \partial_x \kappa_t(x) - \kappa_t(x) \frac{\pa_x \rho_t(x)}{\rho_t(x)} = -\frac{\pa_\eta \kappa_t(\eta)}{\pa_\eta \chi_t(\eta)} + \kappa_t(\eta)\frac{\pa_{\eta\eta}\chi_t(\eta)}{(\pa_\eta \chi_t(\eta))^2}.
\]
On the other hand, Corollary~\ref{cor:propChi} implies 
\[
\mu_t = \lambda(\chi_t - \m_f[\chi_t]) \quad \mbox{and} \quad \kappa_t = (\sigma^2/2)|\chi_t - \m_f[\chi_t]|^2,
\]
respectively, where $\m_f[\chi_t]$ is given by
\[
\m_f[\chi_t] =  \frac{\int_{0}^1 \chi_t \exp(-\alpha f(\chi_t(\eta))) \,d \eta}{\int_{0}^1 \exp(-\alpha f(\chi_t(\eta))) \,d \eta}.
\]
Hence, the pseudo-inverse distribution $\chi_t$ satisfies the following integro-differential equation:
\begin{equation*}
 \partial_t \chi_t + \mu_t = - \partial_\eta\lt(\kappa_t (\partial_\eta \chi_t)^{-1}\rt).
\end{equation*}

Clearly, the results of Section~\ref{sec_lt} concerning the concentration and approximation of the global minimizer hold in the equivalent formulation in terms of $\chi_t$ as well. 


\subsection{Porous media version of the evolution equation}
One very common application of the pseudo-inverse distribution $\chi_t$ is to study the behavior of the support $\supp(\rho_t)$ of the corresponding probability measure $\rho_t$. This is especially interesting when $\rho_t$ has compact support. Unfortunately, we do not have that in the present case due to the diffusion, which causes $\rho_t$ to have full support in $\R$. This naturally leads to the idea of increasing the power of $\rho_t$ in the diffusion term, inspired by the porous media equation.\cite{carrillo2004finite} The evolution equation for $\rho_t$ then becomes
\begin{align}\label{eq:porousRho}
 \partial_t \rho_t + \partial_x(\mu_t \rho_t) = \partial_{xx} (\kappa_t \rho_t^p).
\end{align}
with porous media coefficient $p \ge 1$. Notice that the previous model is included here for $p=1$. The derivation of the evolution equation for $\chi_t$ corresponding to this equation may be analogously done, which leads to 
\begin{align}\label{eq:porousChi}
\partial_t \chi_t + \mu_t = -\partial_{\eta}(\kappa_t (\partial_\eta \chi_t)^{-p}). 
\end{align}
Further investigation of the diffusion term results in
\[
 \partial_\eta(\kappa_t (\partial_\eta \chi_t)^{-p} ) = \partial_\eta (\kappa_t\, \rho_t^p) = \rho_t^p \partial_\eta \kappa_t + \kappa_t \partial_\eta (\rho_t^p) = \rho_t^p \partial_\eta \kappa_t + p\kappa_t \rho_t^{p-1} \partial_\eta \rho_t
\]
in $(\eta,t)$ variables. For $p > 1$ we can do the following formal computations. Due to mass conservation of $\rho_t$ we assume a no flux condition for \eqref{eq:porousRho} which in $(x,t)$ variables reads
\[
 \mu_t\rho_t = \partial_x (\kappa_t \rho_t^p) = \rho_t^p \partial_x \kappa_t + p\kappa_t \rho_t^{p-1} \partial_x \rho_t,
\]
on the boundary points of $\supp(\rho_t)$. Consequently, we obtain
\begin{align*}
 \partial_\eta(\kappa_t (\partial_\eta \chi_t)^{-p})(F_t(x)) &= [\rho_t^{p}  \partial_x \kappa_t + p \kappa_t \rho_t^{p-1} \partial_x \rho_t]\partial_\eta \chi_t (F_t(x))\\
 &= \mu_t \rho_t \partial_\eta \chi_t (F_t(x)) = \mu_t,
\end{align*}
on the boundary points of $\supp(\rho_t)$. Therefore, restricting \eqref{eq:porousChi} onto the boundary points yields
\begin{equation}\label{eq:evolutionBoundary}
 \partial_t \chi_t (\eta) = - 2 \mu_t(\eta) = -2 \lambda (\chi_t(\eta) - \m_f[\chi_t]) \qquad \text{ for } \eta \in \{0,1\}.
\end{equation}
As $\m_f[\chi_t]$ is contained in the interior of $\supp(\rho_t)$ by definition, $\mu_t$ is negative at the left boundary point $\eta=0$ and positive at the right boundary point $\eta=1$. Hence, \eqref{eq:evolutionBoundary} implies the shrinking of $\supp(\rho_t)$. In particular, one expects the concentration of $\chi_t$ at $\m_f$ as $t\to\infty$. We will numerical check this behavior in the next subsection.

\subsection{Discretization of the evolution equation for $\chi_t$}
To investigate the behavior of the pseudo-inverse $\chi_t$ numerically, we use an implicit finite difference scheme. Following the ideas in \cite{blanchet2008convergence} we denote the discretized version of $\chi_t$ by $\chi_k^i$, where the spatial discretization is indexed by $k$ and the temporal discretization by $i$. The spatial and temporal step sizes are denoted by $h$ and $\tau$, respectively. A straight forward discretization of the general equation \eqref{eq:porousChi} yields
\begin{equation}\label{disretization_porous}
 \frac{\chi_{k}^{i+1} - \chi_{k}^i}{\tau} = -  \left( \frac{\kappa (\chi_{k+1}^{i+1},\m_f^{i})}{(\chi_{k+1}^{i+1}-\chi_{k}^{i+1})^p}-\frac{\kappa (\chi_{k}^{i+1},\m_f^{i})}{(\chi_{k}^{i+1}-\chi_{k-1}^{i+1})^p} \right) h^{p-1} + \lambda (\m_f^{i} - \chi_{k}^{i+1}),
\end{equation}
for $\eta \in (0,1)$, where $\m_f^i = \m_f[\chi^i]$. At the boundary points $\eta=0,1$ the expressions
\[
 \big(\chi_{k}^{i}-\chi_{k-1}^{i}\big)^{-p} \quad \text{and} \quad \big(\chi_{k+1}^{i}-\chi_{k}^{i}\big)^{-p}
\]
are set to zero, respectively. As stopping criterion for the iteration procedure we use 
\begin{equation*}
 \|\chi^{i+1} - \chi^i\|_{L^2(0,1)} < \tol.
\end{equation*}
Since we expect the density $\rho_t$ to concentrate close to the minimizer $x_*\in\R^d$ of the cost function $f$, the pseudo-inverse $\chi_t$ should converge towards the constant function with value $x_*\in\R^d$. This causes problems in the computation of the fractions appearing in \eqref{disretization_porous}. Our workaround is to evaluate the fractions up to a tolerance and set them artificially to zero if the denominator is too small. The scheme is tested with the well-known Ackley benchmark function for global optimization problems (see e.g.~\cite{askari2013large}) shown in Figure \ref{fig:Preprocessing}.

\subsection{Particle approximation}
To compare the results of the extension $p>1$ to the scheme in \cite{doi:10.1142/S0218202517400061}, we are interested in a particle scheme corresponding to the evolution equation for $p=2$. Note that in contrast to the pseudo-inverse distribution case, we are not restricted to one dimension here. We derive a numerical scheme by rewriting \eqref{eq:porousRho} as
\begin{equation*}
 \partial_t \rho_t = - \nabla_x (\mu_t \rho_t) + \Delta (\kappa_t \rho_t^2) = \nabla_x[ - \mu_t \rho_t +\rho_t(\nabla_x (\kappa_t \rho_t) + \kappa_t \nabla_x \rho_t) ].
\end{equation*}
The terms $\nabla_x (\kappa_t \rho_t)$ and $\nabla_x \rho_t$ are mollified in the spirit of \cite{klar2014multiscale} with the help of a mollifier $\varphi_\e$,
\begin{equation*}
 \nabla_x (\kappa_t \rho_t) \approx \nabla_x \varphi_\e \ast (\kappa_t \rho_t) \qquad \text{and} \qquad \nabla_x \rho_t \approx  \nabla_x \varphi_\e \ast \rho_t.
\end{equation*}
Altogether this yields the approximate deterministic microscopic system
\begin{equation}\label{eq:porousParticle}
 \dot X_t^i = - \lambda(X_t^i - \m_t) + \frac{\sigma}{N} \sum_{j=1}^N \nabla_x \varphi_\e(X_t^i - X_t^j) \Big[|X_t^j - \m_t|^{2p} + |X_t^i - \m_t|^{2p}\Big], 
\end{equation}
for $i=1,\ldots N$, using the notation in Introduction.
\begin{remark}
 Note that the scheme \eqref{eq:porousParticle} is deterministic in contrast to the scheme \eqref{eq:particle} for $p=1$. Unfortunately, it is not trivial to extend the particle scheme to $p > 2$.
\end{remark}

\subsection{Numerical results}
In the following, numerical results corresponding to the above discretizations are reported. We use $200$ grid points for the spatial discretization of $\chi_t$ and $500$ particles for the particle approximation schemes. Further parameters are fixed as 
\[
 \tau = 2.5\cdot10^{-3},\qquad \alpha=30,\qquad \sigma = 0.8,\qquad \tol = 10^{-6}.
\]
The mollifier is chosen to be $\varphi_\e=\e^{-d}\varphi(x/\e)$, where
\[
 \varphi (x) = \frac{1}{Z_d}\begin{cases} \exp\left(\frac{1}{|x|^2 - 1}\right), &\text{if } |x| < 1 \\  0, &\text{else}  \end{cases},
\]
with normalizing constant\, $Z_d$.

Figure~\ref{fig:Progression} shows the progression of $\chi_t$ over time. On the left side the case $p=1$ is depicted. The tails mentioned in the discussion of \eqref{eq:porousRho} can be seen near the boundary. On the right side the diffusion coefficient is $p=2$, in this case no tails occur as expected.

In \cite{doi:10.1142/S0218202517400061}, the following scheme with an approximate Heaviside function was proposed:
\begin{align*}
 dX_t^i = -\lambda (X_t^i-\m_t)\,H_\e(f(X_t^i)-f(\m_t))\,d t + \sqrt{2} \sigma |X^i_t-\m_f| d W_t^i, 
\end{align*}
where $\m_t$ is as given in \eqref{eq:particle_b}.
Initially, the Heaviside function was included to assure that the particles do not concentrate abruptly. This is essential in cases where the weight parameter $\alpha>0$ is chosen too small, thereby yielding a rough approximation of the minimizer at the start of the simulation. In fact, the presence of the Heaviside function prevents particles that attain function values smaller than the function values at the average, i.e., $f(X^i) < f(\m_t)$, from drifting to $\m_t$. In those cases, only the diffusion part is active.

\begin{figure}
	\begin{minipage}{0.49\textwidth}
		\includegraphics[width=.95\textwidth]{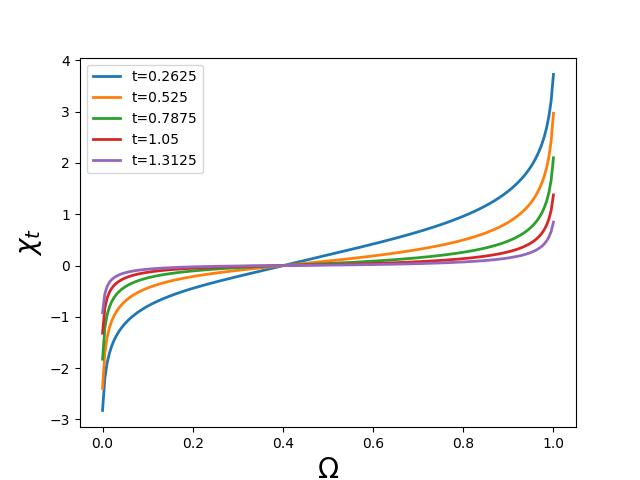}
	\end{minipage}
	\hfill
	\begin{minipage}{0.49\textwidth}
		\includegraphics[width=.95\textwidth]{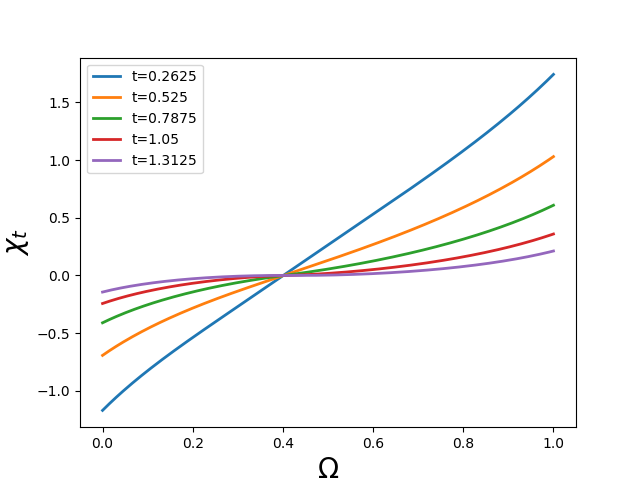}
	\end{minipage}\\[0.5em]
	\caption{Ackley benchmark in 1d. Progression of the inverse distribution function over time for the three benchmarks. Left: Diffusion with $p=1$. Right: Diffusion with $p=2$.}
	\label{fig:Progression}
\end{figure}

An analogous particle scheme for the porous media equation with $p=2$ reads
\begin{align*}
 \dot X_t^i &= - \lambda(X_t^i - \m_t)\,H_\e(f(X_t^i)-f(\m_f)) \\
 &\hspace*{4em}+ \frac{\sigma}{N} \sum_{j=1}^N \nabla_x \varphi(X_t^i - X_t^j) \Big[|X_t^j - \m_t|^{2p} + |X_t^i - \m_t|^{2p}\Big].
\end{align*}
For both schemes a smooth approximation of the Heaviside function of the form, 
\[
 H_\e=(1+\erf(x/\e))/2
\]
is used.  We compare the results with and without the Heaviside function in Figure~\ref{fig:errorPlots}. In these simulations, we see the damping effect of the Heaviside function. The simulations without Heaviside converge faster. Due to the large value of $\alpha$, the minimizer is approximated well by $\m_f$, thus the concentration happens very close to the actual minimum of the objective functions.

The graphs show the $L_2$-distance of ${\bf X}_t$ (left) and $\chi_t$ (right) to the known minimizer $x_*$ or equivalently the 2-Wasserstein distance between the solutions of the mean-field equation and the particle scheme to the global consensus at $\delta_{x_*}$. The schemes with nonlinear diffusion $p=2$ converge faster than their corresponding schemes with linear diffusion. Nevertheless, for practical applications with large number of particles, the scheme with linear diffusion is more reasonable due to shorter computation times. In fact, in each iteration of the scheme \eqref{eq:porousParticle} the convolution of all particles has to be computed. The error of the simulation for $\chi_t$ is smaller then the one for ${\bf X}_t$ at equal times. The linear graphs with respect to the logarithmic scaling of the y-axis in Figure~\ref{fig:errorPlots} indicate the exponential convergence shown in the theoretical section (cf.~Theorem \ref{thm:concentration}). The stochasticity influencing the schemes for $p=1$ can be observed in the graphs in Figure~\ref{fig:errorPlots} (left).

\begin{figure}
  \begin{minipage}{0.49\textwidth}
  \includegraphics[width=.95\textwidth]{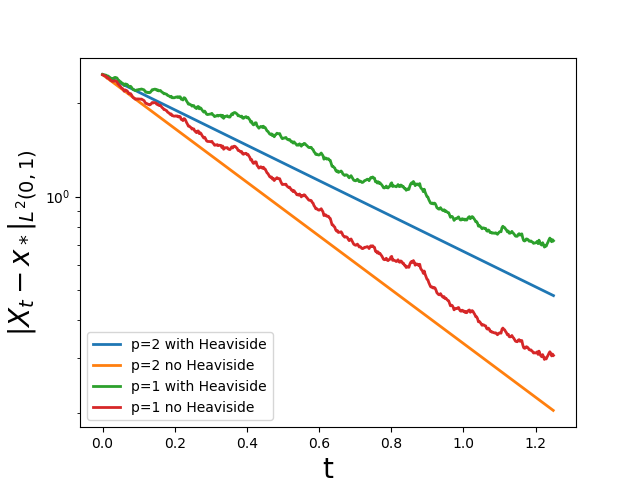}
 \end{minipage}
 \hfill
 \begin{minipage}{0.49\textwidth}
  \includegraphics[width=.95\textwidth]{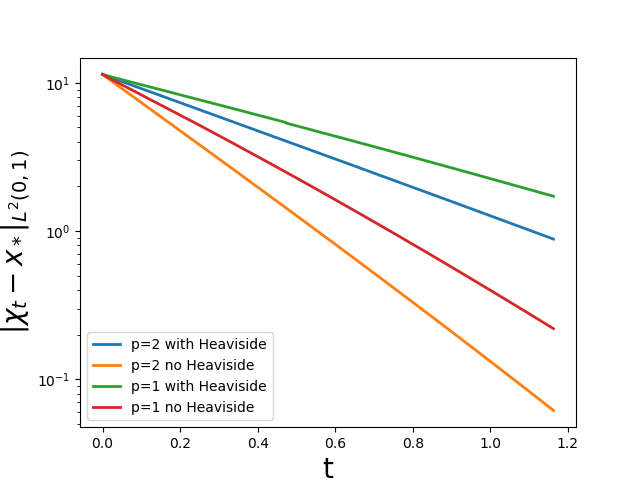}
 \end{minipage}\\[0.5em]
 \caption{Ackley benchmark in one dimension. $L_2$ error of the solution with respect to the minimizer $x_*$ or, equivalently, the 2-Wasserstein distance between the solution and $\delta_{x^*}$ for the different benchmarks. Left: Particle scheme. The stochastic influence is visible for $p=1$, as the graphs rely on data of one realization. Right: Pseudo-inverse distribution. }
 \label{fig:errorPlots}
\end{figure}

In contrast to the pseudo-inverse distribution function, which is only available in one dimension, the particle scheme can be easily generalized to higher dimensions. We conclude the manuscript with some numerical results obtained with the particle scheme in two dimensions applied to the Ackley benchmark. In Figure~\ref{fig:plots2d} (left) we see the surface and contour plot of the Ackley function in two dimensions. In Figure~\ref{fig:plots2d} (right), the convergence of the different particle schemes is illustrated. For the stochastic scheme with $p=1$ the data is averaged over 100 Monte Carlo simulations. The graphs are in good agreement with the corresponding graphs of the pseudo-inverse distribution function in one dimension. Due to the averaging the stochastic influence which can be seen in Figure~\ref{fig:errorPlots} (left) disappears.

\begin{figure}
  \begin{minipage}{0.49\textwidth}
  \includegraphics[width=\textwidth]{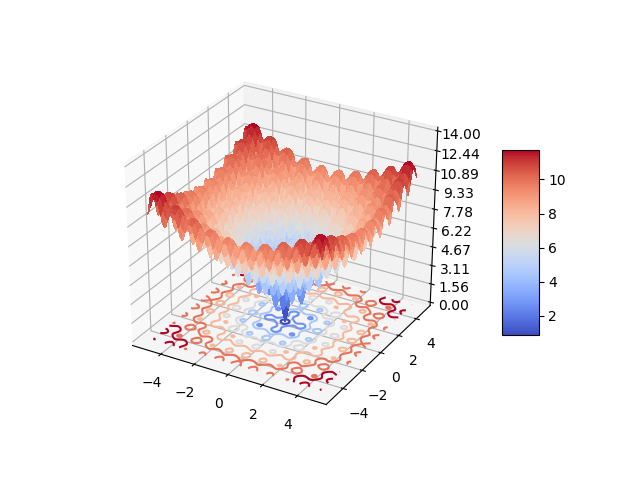}
 \end{minipage}
 \hfill
 \begin{minipage}{0.49\textwidth}
  \includegraphics[width=.95\textwidth]{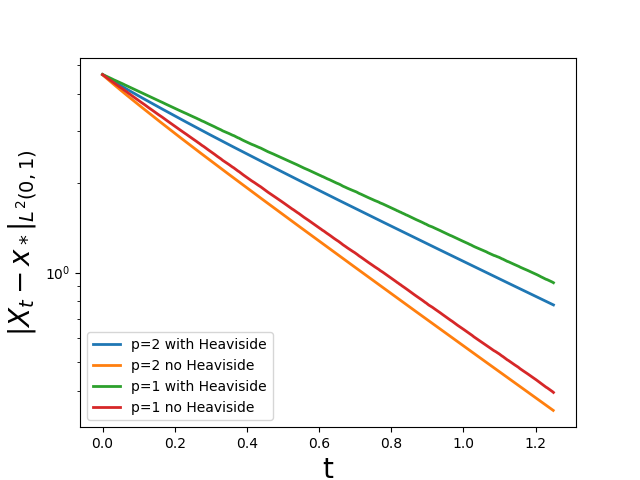}
 \end{minipage}\\[0.5em]
 \caption{Ackley benchmark function in two dimensions. Left: Surface and contour plot of the function in two dimensions. Right: Convergence of the particle schemes. For the stochastic cases with $p=1$ the average of 100 Monte Carlo simulations is shown.}
 \label{fig:plots2d}
\end{figure}

\section*{Acknowledgments}
JAC was partially supported by the Royal Society by a Wolfson Research Merit Award and by EPSRC grant number EP/P031587/1. Y-PC was supported by the Alexander Humboldt Foundation through the Humboldt Research Fellowship for Postdoctoral Researchers. Y-PC was also supported by NRF grants (NRF-2017R1C1B2012918 and 2017R1A4A1014735). CT was partially supported by a 'Kurzstipendium f\"ur Doktorandinnen und Doktoranden' by the German Academic Exchange Service. OT is thankful to Jim Portegies for stimulating discussions.

\bibliographystyle{plain}
\bibliography{consensus}

\begin{thebibliography}{10}

\bibitem{aarts1988simulated}
Emile Aarts and Jan Korst.
\newblock {\em Simulated annealing and Boltzmann machines}.
\newblock New York, NY; John Wiley and Sons Inc., 1988.

\bibitem{arnold1974stochastic}
Ludwig Arnold.
\newblock Stochastic differential equations.
\newblock {\em New York}, 1974.

\bibitem{askari2013large}
Omid Askari-Sichani and Mahdi Jalili.
\newblock Large-scale global optimization through consensus of opinions over
  complex networks.
\newblock {\em Complex Adaptive Systems Modeling}, 1(1):1--18, 2013.

\bibitem{back1997handbook}
Thomas Back, David~B. Fogel, and Zbigniew Michalewicz.
\newblock {\em Handbook of evolutionary computation}.
\newblock IOP Publishing Ltd., 1997.

\bibitem{bellomo2013}
Nicola Bellomo, Abdelghani Bellouquid, and Damian Knopoff.
\newblock From the microscale to collective crowd dynamics.
\newblock {\em Multiscale Modeling \& Simulation}, 11(3):943--963, 2013.

\bibitem{MR3304346}
Andrea~L. Bertozzi, Jes{\'u}s Rosado, Martin~B. Short, and Li~Wang.
\newblock Contagion shocks in one dimension.
\newblock {\em J. Stat. Phys.}, 158(3):647--664, 2015.

\bibitem{bianchi2009survey}
Leonora Bianchi, Marco Dorigo, Luca~Maria Gambardella, and Walter~J Gutjahr.
\newblock A survey on metaheuristics for stochastic combinatorial optimization.
\newblock {\em Natural Computing: an international journal}, 8(2):239--287,
  2009.

\bibitem{blanchet2008convergence}
Adrien Blanchet, Vincent Calvez, and Jos{\'e}~A. Carrillo.
\newblock Convergence of the mass-transport steepest descent scheme for the
  subcritical patlak-keller-segel model.
\newblock {\em SIAM Journal on Numerical Analysis}, 46(2):691--721, 2008.

\bibitem{blum2003metaheuristics}
Christian Blum and Andrea Roli.
\newblock Metaheuristics in combinatorial optimization: Overview and conceptual
  comparison.
\newblock {\em ACM Computing Surveys (CSUR)}, 35(3):268--308, 2003.

\bibitem{BolleyPolish}
Fran{\c{c}}ois Bolley.
\newblock Separability and completeness for the wasserstein distance.
\newblock In {\em S{\'e}minaire de probabilit{\'e}s XLI}, pages 371--377.
  Springer, 2008.

\bibitem{MR3035983}
Fran{\c{c}}ois Bolley, Ivan Gentil, and Arnaud Guillin.
\newblock Uniform convergence to equilibrium for granular media.
\newblock {\em Arch. Ration. Mech. Anal.}, 208(2):429--445, 2013.

\bibitem{MR3331178}
Jos{\'e}~A. Carrillo, Young-Pil Choi, and Maxime Hauray.
\newblock The derivation of swarming models: mean-field limit and {W}asserstein
  distances.
\newblock In {\em Collective dynamics from bacteria to crowds}, volume 553 of
  {\em CISM Courses and Lectures}, pages 1--46. Springer, Vienna, 2014.

\bibitem{MR2596552}
Jos{\'e}~A. Carrillo, Massimo Fornasier, Jes{\'u}s Rosado, and Giuseppe
  Toscani.
\newblock Asymptotic flocking dynamics for the kinetic cucker-smale model.
\newblock {\em SIAM Journal on Mathematical Analysis}, 42(1):218--236, 2010.

\bibitem{MR2744704}
Jos{\'e}~A. Carrillo, Massimo Fornasier, Giuseppe Toscani, and Francesco Vecil.
\newblock Particle, kinetic, and hydrodynamic models of swarming.
\newblock In {\em Mathematical modeling of collective behavior in
  socio-economic and life sciences}, Model. Simul. Sci. Eng. Technol., pages
  297--336. Birkh\"auser Boston, Inc., Boston, MA, 2010.

\bibitem{carrillo2004finite}
Jos{\'e}~A. Carrillo, Maria~Pia Gualdani, and Giuseppe Toscani.
\newblock Finite speed of propagation in porous media by mass transportation
  methods.
\newblock {\em Comptes Rendus Mathematique}, 338(10):815--818, 2004.

\bibitem{MR3251743}
Jos{\'e}~A Carrillo, Yanghong Huang, and Stephan Martin.
\newblock Explicit flock solutions for quasi-morse potentials.
\newblock {\em European Journal of Applied Mathematics}, 25(05):553--578, 2014.

\bibitem{MR3158478}
Jos{\'e}~A Carrillo, Yanghong Huang, and Stephan Martin.
\newblock Nonlinear stability of flock solutions in second-order swarming
  models.
\newblock {\em Nonlinear Anal. Real World Appl.}, 17:332--343, 2014.

\bibitem{MR2295620}
Felipe Cucker and Steve Smale.
\newblock On the mathematics of emergence.
\newblock {\em Jpn. J. Math.}, 2(1):197--227, 2007.

\bibitem{dembo2009large}
Amir Dembo and Ofer Zeitouni.
\newblock {\em Large deviations techniques and applications}, volume~38.
\newblock Springer Science \& Business Media, 2009.

\bibitem{dorigo2005ant}
Marco Dorigo and Christian Blum.
\newblock Ant colony optimization theory: A survey.
\newblock {\em Theoretical computer science}, 344(2):243--278, 2005.

\bibitem{durrett1996stochastic}
Richard Durrett.
\newblock {\em Stochastic calculus: a practical introduction}, volume~6.
\newblock CRC press, 1996.

\bibitem{gilbarg2015elliptic}
David Gilbarg and Neil~S Trudinger.
\newblock {\em Elliptic partial differential equations of second order}.
\newblock Springer, 2015.

\bibitem{MR2425606}
Seung-Yeal Ha and Eitan Tadmor.
\newblock From particle to kinetic and hydrodynamic descriptions of flocking.
\newblock {\em Kinet. Relat. Models}, 1(3):415--435, 2008.

\bibitem{holley1988simulated}
Richard Holley and Daniel Stroock.
\newblock Simulated annealing via sobolev inequalities.
\newblock {\em Comm. Math. Phys.}, 115(4):553--569, 1988.

\bibitem{holley1989asymptotics}
Richard~A Holley, Shigeo Kusuoka, and Daniel~W Stroock.
\newblock Asymptotics of the spectral gap with applications to the theory of
  simulated annealing.
\newblock {\em Journal of functional analysis}, 83(2):333--347, 1989.

\bibitem{kennedy2010particle}
James Kennedy.
\newblock Particle swarm optimization.
\newblock In {\em Encyclopedia of Machine Learning}, pages 760--766. Springer,
  2010.

\bibitem{klar2014multiscale}
Axel Klar and Sudarshan Tiwari.
\newblock A multiscale meshfree method for macroscopic approximations of
  interacting particle systems.
\newblock {\em Multiscale Modeling \& Simulation}, 12(3):1167--1192, 2014.

\bibitem{MR3143990}
Theodore Kolokolnikov, Jos{\'e}~A. Carrillo, Andrea Bertozzi, Razvan Fetecau,
  and Mark Lewis.
\newblock Emergent behaviour in multi-particle systems with non-local
  interactions [{E}ditorial].
\newblock {\em Phys. D}, 260:1--4, 2013.

\bibitem{MR3274797}
Sebastien Motsch and Eitan Tadmor.
\newblock Heterophilious dynamics enhances consensus.
\newblock {\em SIAM Rev.}, 56(4):577--621, 2014.

\bibitem{MR2086916}
Reza Olfati-Saber and Richard~M. Murray.
\newblock Consensus problems in networks of agents with switching topology and
  time-delays.
\newblock {\em IEEE Trans. Automat. Control}, 49(9):1520--1533, 2004.

\bibitem{doi:10.1142/S0218202517400061}
Ren\'e Pinnau, Claudia Totzeck, Oliver Tse, and Stephan Martin.
\newblock A consensus-based model for global optimization and its mean-field
  limit.
\newblock {\em Math. Mod. Meth. Appl. Sci.}, 27(01):183--204, 2017.

\bibitem{reeves2003genetic}
Colin Reeves.
\newblock {\em Genetic algorithms}.
\newblock Springer, 2003.

\bibitem{sznitman1991topics}
Alain-Sol Sznitman.
\newblock Topics in propagation of chaos.
\newblock In {\em Ecole d'{\'e}t{\'e} de probabilit{\'e}s de Saint-Flour
  XIX—1989}, pages 165--251. Springer, 1991.

\bibitem{MR2247927}
Giuseppe Toscani.
\newblock Kinetic models of opinion formation.
\newblock {\em Commun. Math. Sci.}, 4(3):481--496, 2006.

\bibitem{villani2003topics}
C{\'e}dric Villani.
\newblock {\em Topics in Optimal Transportation (Graduate Studies in
  Mathematics, Vol. 58)}.
\newblock American Mathematical Society, 2003.

\bibitem{OldandNew}
C{\'e}dric Villani.
\newblock {\em Optimal transport: old and new}, volume 338.
\newblock Springer Science \& Business Media, 2008.

\bibitem{MR3045811}
James von Brecht, Theodore Kolokolnikov, Andrea~L. Bertozzi, and Hui Sun.
\newblock Swarming on random graphs.
\newblock {\em J. Stat. Phys.}, 151(1-2):150--173, 2013.

\end{thebibliography}

\end{document}